\RequirePackage{fix-cm}

\documentclass[smallextended]{svjour3}       
\smartqed  
\usepackage{graphicx}

\usepackage{amsmath} 
\usepackage{amssymb}

\usepackage{pdfsync}

\newtheorem{Satz}{Theorem} 
\newtheorem{Lemm}{Lemma} 
\newtheorem{Def}{Definition} 

\newcommand \Id  {\operatorname{Id}}

\newcommand \R   {\mathbb{R}}
\newcommand \C   {\mathbb{C}}

\newcommand \K   {\mathcal{K}}
\newcommand \Kinf{\mathcal{K_\infty}}
\newcommand \KL  {\mathcal{KL}}
\newcommand \LL  {\mathcal{L}}

\newcommand \T   {\mathcal{T}}

\newcommand \Iff   {\Leftrightarrow}

\newcommand{\lel}{\left\langle}
\newcommand{\rir}{\right\rangle}

\newcommand \UGx  {0-UGAS$x$}
\newcommand \eps {\varepsilon}

 \journalname{Mathematics of Control, Signals, and Systems}

\begin{document}

\title{Input-to-state stability of infinite-dimensional control systems
}

\titlerunning{ISS of infinite-dimensional systems}        


\author{Sergey Dashkovskiy     \and      Andrii Mironchenko }



\institute{Sergey Dashkovskiy \at
							Department of Civil Engineering, University of Applied Sciences Erfurt, Erfurt, Germany\\
              \email{sergey.dashkovskiy@fh-erfurt.de}            
           \and
           Andrii Mironchenko (corresponding author) \at
              Faculty of Mathematics and Computer Science, University of Bremen, Bibliothekstra\ss e 1, 28334 Bremen, Germany \\
              Tel.: +49-421-21863825\\
              Fax: +49-421-2189863825\\
              \email{andmir@math.uni-bremen.de}           
}

\date{Received: date / Accepted: date}

\maketitle

\begin{abstract}
We develop tools for investigation of input-to-state stability (ISS) of infinite-dimensional control systems.
We show that for certain classes of admissible inputs the existence of an ISS-Lyapunov function implies the input-to-state stability of a system.
Then for the case of systems described by abstract equations in Banach spaces we develop two methods of construction of local and global ISS-Lyapunov functions.
We prove a linearization principle that allows a construction of a  local ISS-Lyapunov function for a system which linear approximation is ISS.
In order to study interconnections of nonlinear infinite-dimensional systems, we generalize the small-gain theorem to the case of infinite-dimensional systems and provide a way to construct an ISS-Lyapunov function for an entire interconnection, if ISS-Lyapunov functions for subsystems are known and the small-gain condition is satisfied. We illustrate the theory on examples of linear and semilinear reaction-diffusion equations.
\end{abstract}



\section{Introduction}

The concept of input-to-state stability (ISS) introduced in \cite{Son89} is widely used to study stability properties of control systems with respect to external inputs. 
Within last two decades different methods for verification of the input-to-state stability of finite-dimensional systems were developed. For a survey of recent results in the ISS theory see \cite{Son08} and \cite{DES11}, the relation of ISS and circle criterion can be found in \cite{JLR11}.
In particular it is known that the method of Lyapunov functions together with small-gain theorems (see \cite{JMW96}, \cite{DRW07}, \cite{DRW10}, \cite{KaJ11}) provides us with rich tools to investigate input-to-state stability of control systems.

To input-to-state stability of infinite-dimensional systems, with an exception of time-delay systems (see, e.g. \cite{PeJ06}, \cite{KaJ11}), less attention was devoted. 

In \cite{MaP11} ISS of certain classes of semilinear parabolic equations has been studied with the help of strict Lyapunov functions. In \cite{PrM12} a construction of ISS-Lyapunov functions for certain time-variant linear systems of hyperbolic equations (balance laws) has been provided.
In \cite{DaM10} it was shown, that for certain classes of monotone reaction-diffusion systems ISS of the system with diffusion follows from the ISS of its local dynamics (i.e. of a system without diffusion). 

Other results have been obtained for general control systems via vector Lyapunov functions. 
In \cite{KaJ11} a general vector Lyapunov small-gain theorem for abstract control systems satisfying weak semigroup property (see also \cite{Kar07}, \cite{KaJ11b}) has been proved. For this class of systems in \cite{KaJ10} the trajectory-based small-gain results have been obtained and applied to a chemostat model.

In \cite{JLR08} the results on relations between circle-criterion and ISS for systems, based on equations in Banach spaces, have been proved.

Our guideline in this paper is a development of Lyapunov-type sufficient conditions for ISS of the
infinite-dimensional systems and elaboration of methods for construction of ISS-Lyapunov functions.

Our first main result is that for abstract control systems under certain assumptions on the class of input functions from the existence of a (local or global) ISS-Lyapunov function it follows (local or global) ISS of the system. We show that our definition of the local ISS-Lyapunov function is consistent with the standard definition of local ISS-Lyapunov function for finite-dimensional systems.

In the next part of the paper we exploit semigroup theory methods and consider infinite-dimensional systems generated by differential equations in abstract spaces.
For such systems we develop two methods for construction of ISS-Lyapunov functions for  control systems.

To study interconnections of $n$ ISS subsystems, we generalize small-gain theorem for finite-dimensional systems \cite{DRW06b}, \cite{DRW10} to the infinite-dimensional case.
This theorem allows a construction of a ISS-Lyapunov function for the whole interconnection if ISS-Lyapunov functions for subsystems are known and the small-gain condition is satisfied. The ISS of the interconnection follows then from the existence of an ISS-Lyapunov function for it.

The local ISS of nonlinear control systems can be investigated in an analogous way (see, e.g.,  \cite{DaR10}), but also another type of results is possible, namely linearization technique, well-known for infinite-dimensional dynamical systems (without inputs) \cite{Hen81}. 
We prove, that a system is LISS provided its linearization is ISS in two ways.  The first proof holds for systems with a Banach state space, but it doesn't provide a LISS-Lyapunov function. Another proof is based on a converse Lyapunov theorem and provides a LISS-Lyapunov function, but needs that the state space is Hilbert. 

Throughout the paper we use either classical solutions of partial differential equations, or the solutions in the Sobolev spaces. Other function spaces can be also exploited, see e.g. \cite{Dac03}.

The outline of the work is as follows: in Section \ref{Prelim} we introduce notation and basic notions. In Section \ref{SectLinSys} we discuss ISS for linear systems. Afterwards the method of ISS-Lyapunov functions is extended to the abstract control systems and the results are applied to certain nonlinear reaction-diffusion equation.
In Section \ref{Linearisierung} we prove a linearization principle. 
Next, in Section \ref{GekoppelteISS_Systeme} we prove a small-gain theorem and apply it to certain linear and nonlinear systems.

\section{Preliminaries}
\label{Prelim}

Throughout the paper let $(X,\|\cdot\|_X)$ and $(U,\|\cdot\|_U)$ be the state space and the space of input values, endowed with norms $\|\cdot\|_X$  and $\|\cdot\|_U$ respectively. 

For linear normed spaces $X,Y$ let $L(X,Y)$ be the spaces of bounded linear operators from $X$ to $Y$ and $L(X):=L(X,X)$. A norm in these spaces we denote by $\| \cdot \|$. 

By $C(X,Y)$ we denote the space of continuous functions from $X$ to $Y$, $C(X):=C(X,X)$ and by 
$PC(X,Y)$ the space of piecewise right-continuous functions from $X$ to $Y$. Both are equipped with the standard $\sup$-norm.

Let $\R_+:=[0,\infty)$. 
We will use throughout the paper the following function spaces: 
\begin{itemize}
  \item $C_0^k(0,d)$ is a space of $k$ times continuously differentiable functions \\
	$f:(0,d) \to \R$ with a support, compact in $(0,d)$.
	\item $L_p(0,d)$, $p \geq 1$ is a space of $p$-th power integrable functions $f:(0,d) \to \R$ with the norm $\|f\|_{L_p(0,d)}=\left( \int_0^d{|f(x)|^p dx} \right)^{\frac{1}{p}}$.
	\item $W^{p,k}(0,d)$ is a Sobolev space of functions $f \in L_p(0,d)$, which have weak derivatives of order $\leq k$, all of which belong to $L_p(0,d)$.
	Norm in $W^{p,k}(0,d)$ is defined by $\|f\|_{W^{p,k}(0,d)}=\left( \int_0^d{\sum_{1 \leq s \leq k}\left|\frac{\partial^s f}{\partial x^s}(x)\right|^p dx} \right)^{\frac{1}{p}}$.
	\item $W^{p,k}_0(0,d)$ is a closure of $C_0^k(0,d)$ in the norm of $W^{p,k}(0,d)$.
	\item $H^k(0,d)=W^{2,k}(0,d)$,  $H^k_0(0,d)=W^{2,k}_0(0,d)$.
\end{itemize}

We use the following axiomatic definition of a control system
\begin{Def}
\label{Steurungssystem}
The triple $\Sigma=(X,U_c,\phi)$, consisting of the state space $X$, the space of admissible input functions $U_c \subset \{f:\R_+ \to U\}$, both of which are linear normed spaces, equipped with norms $\|\cdot\|_{X}$ and $\|\cdot\|_{U_c}$ respectively  and of a transition map $\phi:A_{\phi} \to X$, $A_{\phi} \subset \R_+ \times \R_+ \times X \times U_c$ is called a control system, if the following properties hold:
\begin{itemize}
	\item Existence: For every $(t_0,\phi_0,u) \in \R_+ \times X \times U_c$ there exists $t>t_0$: $[t_0,t] \times \{(t_0,\phi_0,u)\} \subset A_{\phi}$.
	\item Identity property: For every $(t_0,\phi_0) \in \R_+ \times X$ it holds $\phi(t_0,t_0, \phi_0,\cdot)~=~\phi_0$.
	\item Causality: For every $(t,t_0,\phi_0,u) \in A_{\phi}$, for every $\tilde{u} \in U_c$, such that $u(s) = \tilde{u}(s)$, $s \in [t_0,t]$ it holds $(t,t_0,\phi_0,\tilde{u}) \in A_{\phi}$ and $\phi(t,t_0,\phi_0,u) \equiv \phi(t,t_0,\phi_0,\tilde{u})$.
	\item Continuity: for each $(t_0,\phi_0,u) \in \R_+ \times X \times U_c$ the map $t \to \phi(t,t_0,\phi_0,u)$ is continuous.
	\item Semigroup property: for all $t \geq s \geq 0$, for all $\phi_0 \in X$, $u \in U_c$ so that $(t,s,\phi_0,u) \in A_{\phi}$, it follows
\begin{itemize}
	\item $(r,s,\phi_0,u) \in A_{\phi}$, $r \in [s,t]$,
	\item for all $r \in [s,t]$ it holds $\phi(t,r,\phi(r,s,x,u),u)=\phi(t,s,x,u)$.
\end{itemize}
%
	%
\end{itemize}
\end{Def}
Here $\phi(t,s,x,u)$ denotes the state of a system at the moment $t \in \R_+$, if its state at the moment $s \in \R_+$ was $x \in X$ and the input $u \in U_c$ was applied.

This definition is adopted from \cite{KaJ11}, but we specialize it to the systems, which satisfy classical semigroup property. Another axiomatic definitions of control systems are also used in the literature (see \cite{Son98a}, \cite{Wil72}).

We assume throughout the paper, that for control systems a BIC property (Boundedness-Implies-Continuation property) holds (see \cite[p. 4]{KaJ11b}, \cite{KaJ11}): for all
$(t_0,x_0,u) \in \R_+ \times X \times U_c$ there exist maximal time of existence of the solution $t_{m} \in (t_0,\infty]$, such that $[t_0,t_m) \times \{(t_0,x_0,u)\} \subset A_{\phi}$ and for all $t \geq t_m$ $(t,t_0,x_0,u) \notin A_{\phi}$. Moreover, if $t_m < \infty$, then for every $M>0$ there exists $t \in [t_0,t_m)$: $\|\phi(t,t_0,x,u)\|_X>M$.

In other words, BIC property states that a solution may stop to exist in finite time only because of blow-up phenomena, when the norm of a solution goes to infinity in finite time.
As examples in this paper we use the parabolic systems, for which the BIC property holds, because of smoothing action of parabolic systems, see \cite{Hen81}.

In this paper we consider time-invariant systems.
Time-invariance means, that the future evolution of a system depends only on the initial state of the system and on the applied input, but not on the initial time. 
For time-invariant systems we can without restriction assume that initial time $t_0:=0$. We denote for short $\phi(t,\phi_0,u):=\phi(t,0,\phi_0,u)$.


\begin{Def}
For the formulation of stability properties the following classes of functions are useful:
\begin{equation*}
\begin{array}{ll}
{\K} &:= \left\{\gamma:\R_+ \to \R_+ \left|\ \right. \gamma\mbox{ is continuous and strictly increasing, }\gamma(0)=0\right\}\\
{\K_{\infty}}&:=\left\{\gamma\in\K\left|\ \gamma\mbox{ is unbounded}\right.\right\}\\
{\LL}&:=\left\{\gamma:\R_+ \to \R_+ \left|\ \gamma\mbox{ is continuous and decreasing with}\right.
 \lim\limits_{t\rightarrow\infty}\gamma(t)=0 \right\}\\
{\KL} &:= \left\{\beta: \R_+^2 \to \R_+ \left|\ \right. \beta(\cdot,t)\in{\K},\ \forall t \geq 0,\  \beta(r,\cdot)\in {\LL},\ \forall r >0\right\}
\end{array}
\end{equation*}
\end{Def}


\begin{Def}
$\Sigma$ is {\it globally asymptotically
stable at zero} uniformly with respect to $x$ (0-UGAS$x$), if $\exists \beta \in \KL$, such that $\forall \phi_0 \in X$, $\forall t\geq 0$ it holds
\begin{equation}
\label{UniStabAbschaetzung}
\left\| \phi(t,\phi_0,0) \right\|_{X} \leq  \beta(\left\| \phi_0 \right\|_{X},t) .
\end{equation}

If $\beta$ can be chosen as $\beta(r,t)=Me^{-at} r $ $\forall r,t \in \R_+$, for some $a,M>0$, then 
$\Sigma$ is called exponentially 0-UGAS$x$.
\end{Def}

The notion 0-UGAS$x$ is also called uniform asymptotic stability in the whole (see \cite[p. 174]{Hah67}).

We need also another notion:
\begin{Def}
$\Sigma$ is {\it globally asymptotically
stable at zero} (0-GAS), if it holds
\begin{enumerate}
	\item $\forall \varepsilon>0 \ \exists \delta>0: \|x\|_X < \delta,\ t\geq 0 \Rightarrow \|\phi(t,x,0)\|_X < \varepsilon$,
	\item  $\forall x \in X$ $\| \phi(t,x,0)\|_X \to 0,\ t \to \infty$.
\end{enumerate}
\end{Def}

In other words, $\Sigma$ is 0-GAS if it is locally stable and globally attractive (see, e.g. \cite{SoW96}).
Note, that the 0-UGAS$x$ property is not equivalent to the 0-GAS in general (\cite[§ 36]{Hah67}, see also Section \ref{Gegenbeispiel}).

\begin{Def}
Element of state space $\phi_0 \in X$ is called an equilibrium point of control system $\Sigma$ if $\phi(t,\phi_0,0) = \phi_0$, for all $t\geq 0$.
\end{Def}

To study stability properties of control systems with respect to external inputs, we introduce the following notion
\begin{Def}
$\Sigma$ is called {\it locally input-to-state stable}
(LISS), if $\exists \rho_x, \rho_u>0$ and  $\exists \beta \in \KL$ and $\gamma \in \K$, such that
the inequality
\begin {equation}
\label{iss_sum}
\begin {array} {lll}
\| \phi(t,\phi_0,u) \|_{X} \leq \beta(\| \phi_0 \|_{X},t) + \gamma( \|u\|_{U_c})
\end {array}
\end{equation}
holds $\forall \phi_0: \|\phi_0\|_X \leq \rho_x$,  
$\forall t\geq 0$ and
$\forall u\in U_c$: $\|u\|_{U_c} \leq \rho_u$.

If $\beta$ can be chosen as $\beta(r,t)=Me^{-at} r $ $\forall r,t \in \R_+$, for some $a,M>0$, then $\Sigma$ is called exponentially LISS (eLISS).

The control system is called {\it input-to-state stable}
(ISS), if in the above definition $\rho_x$ and $\rho_u$ can be chosen equal to $\infty$.

If $\Sigma$ is ISS and $\beta$ can be chosen as $\beta(r,t)=Me^{-at} r $ $\forall r,t \in \R_+$, for some $a,M>0$, then $\Sigma$ is called exponentially ISS (eISS).
\end{Def}

One of the most common choices for $U_c$ is the space $U_c:=PC(\R_+,U)$ with the norm $\| \cdot \|_{U_c} := \sup\limits_{0 \leq s \leq \infty} \|u(s)\|_U$.
In this case one can use the alternative definition of the ISS property, which is often used in the literature (see, e.g. \cite{KaJ11}, \cite{HLT08}):
\begin{proposition}
Let $U_c:=PC(\R_+,U)$. Then $\Sigma$ is LISS iff 
$\exists \rho_x, \rho_u>0$ and  $\exists \beta \in \KL$ and $\gamma \in \K$, such that
the inequality
\begin {equation}
\label{iss_sum_equiv}
\begin {array} {lll}
\| \phi(t,\phi_0,u) \|_{X} \leq \beta(\| \phi_0 \|_{X},t) + \gamma( \sup\limits_{0 \leq s \leq t} \|u(s)\|_U)
\end {array}
\end{equation}
holds $\forall \phi_0: \|\phi_0\|_X \leq \rho_x$,  
$\forall t\geq 0$ and
$\forall u\in U_c$: $\|u\|_{U_c} \leq \rho_u$.
\end{proposition}

\begin{proof}
Sufficiency is clear, since $\sup\limits_{0 \leq s \leq t} \|u(s)\|_U \leq \sup\limits_{0 \leq s \leq \infty} \|u(s)\|_U=\|u\|_{U_c}$.

Now let $\Sigma$ be LISS. Due to causality property of $\Sigma$ the state $\phi(\tau,\phi_0,u)$, $\tau \in [0,t]$ of the system $\Sigma$ does not depend on the values of $u(s)$, $s >t$. For arbitrary $t \geq 0$, $\phi_0 \in X$ and $u \in U_c$ consider another input $\tilde{u} \in U_c$, defined by
\[
\tilde{u}(\tau):= 
\left\{ 
\begin{array}{l}
{u(\tau),\tau \in [0,t],} \\
{u(t), \tau >t.}
\end{array}
\right.
\]
The inequality \eqref{iss_sum} holds for all admissible inputs, and hence it holds also for $\tilde{u}$. Substituting $\tilde{u}$ into \eqref{iss_sum} and using that 
$\|\tilde{u}\|_{U_c}= \sup\limits_{0 \leq s \leq t} \|u(s)\|_U$, we obtain \eqref{iss_sum_equiv}. \hfill $\blacksquare$
\end{proof}

The similar property (with $ \mathop{\text{ess} \sup}\limits_{0 \leq s \leq t} \|u(s)\|_U$ instead of $\sup\limits_{0 \leq s \leq t} \|u(s)\|_U$ ) holds for continuous input functions ($U_c:=C(\R_+,U)$), for the class of strongly measurable and essentially bounded inputs $U_c:=L_{\infty}(\R_+,U)$ (which is the standard choice in the case of finite-dimensional systems) and many other classes of input functions.


\section{Linear systems}
\label{SectLinSys}

%

Let $X$ be a Banach space and $\T=\{T(t),\ t \geq 0\}$ be a $C_0$-semigroup on $X$
with an infinitesimal generator $A=\lim\limits_{t \rightarrow +0}{\frac{1}{t}(T(t)x-x)}$.

Consider a linear control system with inputs of the form
\begin{equation}
\label{LinSys}
\begin{array}{l}
{\dot{s}=As+f(u(t)),} \\
s(0)=s_0,
\end{array}
\end{equation}
where $f:U \to X$ is continuous and so that for some $\gamma \in \K$ it holds
\begin{equation}
\label{F_Abschaetzung}
\|f(u)\|_X \leq \gamma(\|u\|_U), \ \forall u \in U.
\end{equation}


We consider weak solutions of the problem \eqref{LinSys}, which are solutions of integral equation, obtained from \eqref{LinSys} by variation of constants formula
\begin{equation}
\label{LinEq_IntegralForm}
s(t)=T(t)s_0 + \int_0^t{T(t-r)f(u(r))dr},
\end{equation}
where $s_0 \in X$.

The space of admissible inputs $U_c$ can be chosen as an arbitrary subspace of a space of strongly measurable functions $f:[0,\infty) \to U$, such that for all $u \in U_c$ the integral in \eqref{LinEq_IntegralForm} exists in the sense of Bochner. If we define $\phi(t,s_0,u):=s(t)$ by the formula \eqref{LinEq_IntegralForm}, we obtain that $(X,U_c,\phi)$ is a control system according to Definition~\ref{Steurungssystem}.

For examples in this section we will use $U_c:=C([0,\infty),U)$.
In this case functions under the sign of integration in \eqref{LinEq_IntegralForm} are strongly measurable (since they are continuous, see \cite[p. 84]{HiP96}) and for all $t \geq 0$
\[
\int_0^t{\|T(t-r)f(u(r))\|_X dr}~<~\infty.
\]
Thus according to the criterion of Bochner integrability (see \cite[Theorem 3.7.4.]{HiP96}), 
integral in \eqref{LinEq_IntegralForm} is well-defined in the sense of Bochner.

The following fact is well-known
\begin{proposition}
\label{Finite_LinGAS_ISS}
For finite-dimensional systems ($X = \R^n$) the following properties of the system \eqref{LinSys} are equivalent: e0-GAS, eISS, 0-GAS, ISS.
\end{proposition}

We are going to obtain a counterpart of this proposition for infinite-dimensional systems. We need the following lemma:
\begin{Lemm}
\label{HilfsAequivalenzen}
The following statements are equivalent:
\begin{enumerate}
	\item \eqref{LinSys} is {\UGx}.
	\item $\T$ is uniformly stable (that is, $\|T(t)\| \to 0,\ t \to \infty$).
	\item $\T$ is uniformly exponentially stable ($\|T(t)\| \leq M e^{-\omega t}$ for some $M,\omega >0$ and all $t \geq 0$).	
	\item \eqref{LinSys} is exponentially {\UGx}.
\end{enumerate}
\end{Lemm}

\begin{proof}
1 $\Iff$ 2. 
At first note that for an input-to-state stable system \eqref{LinSys} $\KL$-function $\beta$ can be always chosen as $\beta(r,t)=\zeta(t) r$ for some $\zeta \in \LL$.
Indeed, consider $x \in X:$ $\|x\|_X=1$, substitute it into \eqref{UniStabAbschaetzung} and choose $\zeta(\cdot)=\beta(1,\cdot) \in \LL$. From linearity of $\T$ we have, that $\forall x \in X$, $x \neq 0$ $\|T(t)x\|_X= \|x\|_X \cdot \|T(t)\frac{x}{\|x\|_X}\|_X \leq \zeta(t) \|x\|_X$. 

Let \eqref{LinSys} be {\UGx}. Then $\exists \zeta \in \LL$, such that 
\[
\|T(t)x\|_X \leq \beta(\|x\|_X, t)=\zeta(t) \|x\|_X  \quad \forall x \in X,\ \forall t \geq 0
\]
holds. This means, that $\|T(\cdot)\| \leq \zeta(\cdot)$, and, consequently, $\T$ is uniformly stable.

If $\T$ is uniformly stable, then it follows, that $\exists \zeta \in \LL$: $\|T(\cdot)\| \leq \zeta(\cdot)$. Then $\forall x \in X$ $\|T(t)x\|_X \leq \zeta(t) \|x\|_X$.

Equivalence 2 $\Iff$ 3 is well-known (see \cite[Proposition 1.2, p. 296]{EnN00}).

3 $\Iff$ 4. Follows from the fact that
for some $M,\omega >0$ it holds that $\|T(t)x\| \leq M e^{-\omega t} \|x\|_X$ $\forall x \in X$  $\Iff$
$\|T(t)\| \leq M e^{-\omega t}$ for some $M,\omega >0$. \hfill $\blacksquare$
\end{proof}


%
%

The following proposition provides us with an infinite-dimensional counterpart of Proposition~\ref{Finite_LinGAS_ISS}
\begin{proposition}
\label{HauptLinSatz}
For systems of the form \eqref{LinSys} it holds:
\[
\text{\eqref{LinSys} is e0-UGAS}x \Iff \text{\eqref{LinSys} is 0-UGAS}x \Iff \text{\eqref{LinSys} is eISS} \Iff \text{\eqref{LinSys} is ISS}. 
\]
\end{proposition}

\begin{proof}
System \eqref{LinSys} is e0-UGAS$x$ $\Iff$ \eqref{LinSys} 0-UGAS$x$  by Lemma \ref{HilfsAequivalenzen}. 

Clearly, from eISS of \eqref{LinSys} it follows ISS of \eqref{LinSys}, and this implies that  \eqref{LinSys} is 0-UGAS$x$ by taking $u \equiv 0$. 
It remains to prove, that {\UGx} of \eqref{LinSys} implies eISS of  \eqref{LinSys}.

Let system \eqref{LinSys} be {\UGx}, then by Lemma \ref{HilfsAequivalenzen}, 
$\T$ is an exponentially stable $C_0$-semigroup, that is, $\exists M,w>0$, such that $\|T(t)\| \leq Me^{-w t}$ for all $t \geq 0$.
From \eqref{LinEq_IntegralForm} we have 
\[
\|s(t)\|_{X} 
             \leq M e^{-wt} \|s_0\|_X + \frac{M}{w} \gamma(\|u\|_{U_c}),
\]
 and the eISS is proved. \hfill $\blacksquare$
\end{proof}

For finite-dimensional linear systems 0-GAS is equivalent to {\UGx} and ISS to eISS, consequently, the Proposition \ref{Finite_LinGAS_ISS} is a special case of 
Proposition~\ref{HauptLinSatz}.

However, for infinite-dimensional linear systems 0-GAS and {\UGx} are not equivalent.
Moreover, 0-GAS in general does not imply bounded-input bounded-state (BIBS) property ($\forall x \in X$, $\forall u\in U_c$: $\|u\|_{U_c} \leq M$ for some $M>0$ $\Rightarrow$ $\|\phi(t,x,u)\|_X \leq R$ for some $R>0$). 
We show this by the following example (another example, which demonstrates this property, can be found in \cite[p.247]{MaP11}).

\subsection{Counterexample}
\label{Gegenbeispiel}

Let $C(\R)$ be the space of continuous functions on $\R$, and let $X=C_0(\R)$ be the Banach space of continuous functions (with sup-norm), that vanish at infinity:
\begin{eqnarray*}
C_0(\R)=\{f \in C(\R) : \forall \varepsilon >0\; \exists \text{ compact set } K_{\varepsilon} \subset \R: |f(s)|<\varepsilon \; \forall s \in \R \backslash K_{\varepsilon}\}.
\end{eqnarray*}


For a given $q \in C(\R)$ consider the multiplication semigroup $T_q$ (for the properties of these semigroups see, e.g., \cite{EnN00}), defined by 
\[
T_q(t)f=e^{tq}f \quad  \forall f \in C_0(\R),
\]
and for all $t\geq 0$ we define $e^{tq}:x \in \R \mapsto e^{tq(x)}$.

Let us take $U=X=C_0(\R)$ and choose $q$ as $q(s)=-\frac{1}{1+|s|}$. Consider the control system, given by
\begin{equation}
\label{GegenBeisp_LinSys}
\dot{x}=A_q x +u,
\end{equation}
where $A_q$ is the infinitesimal generator of $T_q$. 

Let us show, that the system \eqref{GegenBeisp_LinSys} is 0-GAS. Fix arbitrary $f \in C_0(\R)$.
We obtain
\[
\|T_q(t)f\|_{C_0(\R)} = \sup_{s \in \R}  |(T_q(t)f)(s)|= \sup_{s \in \R} e^{-t\frac{1}{1+|s|}}|f(s)| \leq \sup_{s \in \R} |f(s)| = \|f\|_{C_0(\R)}.
\]
This shows that the first axiom of 0-GAS property is satisfied.

To show the global attractivity of the system  note that $\forall \varepsilon >0$ there exists a compact set $K_{\varepsilon} \subset \R$, such that $|f(s)|<\varepsilon \; \forall s \in \R \backslash K_{\varepsilon}$.
For such $\varepsilon$ it holds, that $|(T_q(t)f)(s)|<\varepsilon \; \forall s \in \R \backslash K_{\varepsilon}$, $\forall t \geq 0$. Moreover, there exists $t({\varepsilon})$: $|(T_q(t)f)(s)|<\varepsilon$ 
for all $s \in K_{\varepsilon}$ and $t \geq t({\varepsilon})$. 
Overall, we obtain, that for each $f \in C_0(\R)$ and all $\eps >0$ there exist $t(\eps)>0$ such that $\|T_q(t)f\|_{C_0(\R)}<\varepsilon \ \forall t \geq t({\varepsilon})$. This proves, that system \eqref{GegenBeisp_LinSys} is 0-GAS.

%
%
%

Take constant w.r.t. time external input $u \in C_0(\R) $: $u(s)=a\frac{1}{\sqrt{1+|s|}}$, for some $a>0$ and all $s \in \R$. The solution of \eqref{GegenBeisp_LinSys} is given by:
\begin{eqnarray*}
x(t)(s) &=&e^{-t\frac{1}{1+|s|}}x_0 + \int_0^t{e^{-(t-r)\frac{1}{1+|s|}} \frac{a}{\sqrt{1+|s|}}dr} \\
        &=& e^{-t\frac{1}{1+|s|}}x_0 - a \sqrt{1+|s|} (e^{-t\frac{1}{1+|s|}} - 1).
\end{eqnarray*}
We make a simple estimate, substituting $s=t-1$ for $t>1$:
\[
\sup_{s \in \R} a \left|\sqrt{1+|s|} (e^{-t\frac{1}{1+|s|}} - 1) \right| \geq
a \sqrt{t} (1- e^{-1}) \to \infty, \quad t \to \infty.
\]
For all $x_0\in C_0(\R)$ holds $\|e^{-t\frac{1}{1+|s|}}x_0\|_X \to 0$, $t \to \infty$. Thus, 
$\|x(t)\|_X \to \infty$, $t \to \infty$, and the system \eqref{GegenBeisp_LinSys} possesses unbounded trajectories for arbitrary small inputs.

\subsection{Example: linear parabolic equations with Neumann boundary conditions}

In this subsection we investigate input-to-state stability of a system of parabolic equations with Neumann conditions on the boundary.

Let $G$ be a bounded domain in $\R^p$ with smooth boundary $\partial G$, and let $\Delta$ be Laplacian in $G$. 
Let also $F \in C(G \times \R^m,\R^n)$, $F(x,0) \equiv 0$.  

%

Consider a parabolic system
\begin{equation}
\label{LinKineticWithDiffusion}
\left \{\begin {array} {l}
{\frac {\partial s\left (x, t\right)} {\partial t} - \Delta{s} = R {s}+F(x,u(x,t)), \; x \in G, t>0,} \\
{ {s}\left (x, 0\right) = \phi_0 \left (x\right) , x \in G,} \\
{\left. \frac {\partial  {s}} {\partial n} \right |_{\partial G \times \R_+} =0.}
\end {array}
\right.
\end {equation}
Here $\frac {\partial } {\partial n}$ is the normal derivative, $s(x,t) \in \R^n$, $R \in \R^{n \times n}$ and $u \in C(G \times \R_+, \R^m)$ is an external input.


Let $L:C(\overline{G}) \to C(\overline{G})$,  $L=-\Delta$ with
\[
D(L)=\left\{f \in C^2(G) \cap C^1(\overline{G}): L f \in C(\overline{G}),\  {\left. \frac {\partial  {f}} {\partial n} \right |_{\partial G} =0}  \right\}.
\]


Define the diagonal operator matrix $A={diag(-L,\ldots,-L)}$ with $-L$ as diagonal elements and $D(A)=(D(L))^n$. The closure $\overline{A}$ of $A$ is an infinitesimal generator of an analytic semigroup on $X=(C(\overline{G}))^n$.

Define the space of input values by $U:=C(\overline{G},\R^m)$ and the space of input functions by $U_c:=C(\R_+,U)$.

The problem \eqref{LinKineticWithDiffusion} may be considered as an abstract differential equation:
\[
\dot{s}=(\overline{A}+R)s + f(u(t)),
\]
\[
s(0)=\phi_0,
\]
where $u \in U_c$, $u(t)(x)=u(x,t)$ and $f:U \to X$ is defined by $f(v)(x):=F(x,v(x))$.

One can check, that the map $t \mapsto f(u(t))$ is continuous, and that
\[
\|f(u)\|_X = \sup_{x \in \overline{G}} |f(u)(x)| = \sup_{x \in \overline{G}} |F(x,u(x))| \leq
\gamma(\|u\|_U),
\]
where $\gamma(\|u\|_U):= \sup_{x \in \overline{G},y: |y| \leq \|u\|_U} |F(x,y)|$.

Consequently we have reformulated the problem \eqref{LinKineticWithDiffusion} in the form \eqref{LinSys}.
Note that $\overline{A}+R$ also generates an analytic semigroup, as a sum of infinitesimal generator of analytic semigroup $\overline{A}$ and bounded operator $R$.

Our claim is: 
\begin{proposition}
\label{ParabGleichungenProp}
System \eqref{LinKineticWithDiffusion} is eISS $\Iff$ $R$ is Hurwitz.
\end{proposition}

\begin{proof}
Denote by $S(t)$ the analytic semigroup, generated by $\overline{A} + R$.

We are going to find a simpler representation for $S(t)$. Consider \eqref{LinKineticWithDiffusion} with $u \equiv 0$. Substituting $s(x,t)=e^{Rt}v(x,t)$ in \eqref{LinKineticWithDiffusion}, we obtain a simpler problem for $v$:
\begin{equation}
\label{LinKineticWithDiffusion_Simpl}
\left \{\begin {array} {l}
{\frac {\partial v\left (x, t\right)} {\partial t} = A{v}, \; x \in G, t>0,} \\
{ {v}\left (x, 0\right) = \phi_0 \left (x\right) , x \in G,} \\
{\left. \frac {\partial  {v}} {\partial n} \right |_{\partial G \times \R_+} =0.}
\end {array}
\right.
\end {equation}

In terms of semigroups, it means: $S(t)=e^{Rt}T(t)$, where $T(t)$ is a semigroup generated by $\overline{A}$.
It is well-known (see, e.g. \cite{Hen81}), that the growth bound of analytic semigroup $T(t)$ is given by $\sup{\Re(Spec(\overline{A}))}=\sup_{\lambda \in Spec(\overline{A})} {\Re(\lambda)}$, where $\Re(z)$ is the real part of a complex number $z$. 

We are going to find an upper bound of spectrum of $\overline{A}$ in $D(A)$. Note that $Spec(A)= {Spec(-L)}$. Thus, it is enough to estimate the spectrum of $-L$ that consists of all $\lambda \in \C$, such that the following equation has a nontrivial solution
\begin{equation}
\label{Helmholz}
\left \{\begin {array} {l}
{L{s} + \lambda s =0, \; x \in G} \\
{\left. \frac {\partial  {s}} {\partial n} \right |_{\partial G} =0.}
\end {array}
\right.
\end {equation}

Let $\lambda>0$ be an eigenvalue of $L$ and $u_{\lambda} \not \equiv 0$ be the corresponding eigenfunction. 
If $u_{\lambda}$ attains its nonnegative maximum over $\overline{G}$ in some $x \in G$, then according to the strong maximum principle (see \cite{Eva98}, p. 333)  $u_{\lambda} \equiv const$ and consequently $u_{\lambda} \equiv 0$ $\Rightarrow$ $u_{\lambda}$ cannot be an eigenfunction. 
If $u_{\lambda}$ attains the nonnegative maximum over $\overline{G}$ in some $x \in \partial G$, then by Hopf's lemma (see \cite{Eva98}, p. 330),   $\frac {\partial  {u_{\lambda}}(x)} {\partial n} >0$. 
Consequently, $u_{\lambda} \leq 0$ in $\overline{G}$. But $-u_{\lambda}$ is also an eigenfunction, thus applying the same argument we obtain that $u_{\lambda} \equiv 0$ in $\overline{G}$, thus $\lambda>0$ is not an eigenvalue.  

Obviously $\lambda=0$ is an eigenvalue, therefore growth bound of $T(t)$ is $0$ and growth bound of $S(t)$ is $\omega_0=\sup\{\Re(\lambda): \exists x \neq 0: Rx=\lambda x\}$. Thus, $R$ to be Hurwitz is a sufficient condition for the system \eqref{LinKineticWithDiffusion} to be exponentially 0-UGAS$x$ and, consequently, eISS. 

It is also a necessary condition, because for constant $\phi_0$ and $u \equiv 0$ the solutions of \eqref{LinKineticWithDiffusion} are for arbitrary $x \in G$ the solutions of $\dot{s}=Rs$, and to guarantee the stability of the equilibrium $R$ has to be Hurwitz. \hfill $\blacksquare$
\end{proof}

In \eqref{LinKineticWithDiffusion} the diffusion coefficients are equal to one. In case, when the diffusion coefficients of different subsystems are not equal to each other the statement of Proposition \ref{ParabGleichungenProp} is in general not true because of Turing instability phenomenon (see \cite{Tur52}, \cite{Mur03}).


\section{Lyapunov functions for nonlinear systems}
\label{LyapFunk}

To verify both local and global input-to-state stability of nonlinear systems, Lyapunov functions can be exploited. In this section we provide basic tools and illustrate them by an example.

\begin{Def}
\label{DefLISS_LF}
A continuous function $V:D \to \R_+$, $D \subset X$, $0 \in int(D)=D \backslash \partial D$ is called a local ISS-Lyapunov function (LISS-LF) for $\Sigma$, if $\exists \rho_x,\rho_u>0$ and functions $\psi_1,\psi_2 \in \Kinf$, $\chi \in \K$ and positive definite function $\alpha$, such that:
\begin{equation}
\label{LyapFunk_1Eig}
\psi_1(\|x\|_X) \leq V(x) \leq \psi_2(\|x\|_X), \quad \forall x \in D
\end{equation}
and 
$\forall x \in X: \|x\|_X \leq \rho_x, \; \forall u\in U_c: \|u\|_{U_c} \leq \rho_u$ it holds:
\begin{equation}
\label{GainImplikation}
 \|x\|_X \geq \chi(\|u\|_{U_c}) \  \Rightarrow  \ \dot{V}_u(x) \leq -\alpha(\|x\|_X),
\end{equation}
where the Lie derivative of $V$ corresponding to the input $u$ is given by
\begin{equation}
\label{LyapAbleitung}
\dot{V}_u(x)=\mathop{\overline{\lim}} \limits_{t \rightarrow +0} {\frac{1}{t}(V(\phi(t,x,u))-V(x)) }.
\end{equation}
Function $\chi$ is called ISS-Lyapunov gain for $(X,U_c,\phi)$.

If in the previous definition $D=X$, $\rho_x=\infty$ and $\rho_u=\infty$, then the function $V$ is called ISS-Lyapunov function.
\end{Def}
Note, that in general a computation of the Lie derivative $\dot{V}_u(x)$ requires knowledge of the input on some neighborhood of the time instant $t=0$.

If the input, with respect to which the Lie derivative $\dot{V}_u(x)$ is computed, is clear from the context, then we write simply $\dot{V}(x)$.

\begin{Satz}
\label{LyapunovTheorem}
Let $\Sigma = (X,U_c,\phi)$ be a time-invariant control system, and $x \equiv 0$ be its equilibrium point.

Also let for all $u \in U_c$  and for all $s \geq 0$ a function $\tilde{u}$, defined by $\tilde{u}(\tau)= u(\tau + s)$ for all $\tau \geq 0$, belong to $U_c$ and $\|\tilde{u}\|_{U_c} \leq \|u\|_{U_c}$.

If $\Sigma$ possesses a (L)ISS-Lyapunov function, then it is (L)ISS.
\end{Satz}

For a counterpart of this theorem for infinite-dimensional dynamical systems without inputs see, e.g., \cite{Hen81}.

\begin{proof}
Let the control system $\Sigma=(X,U_c,\phi)$ possess a LISS-Lyapunov function and $\psi_1,\psi_2,\chi,\alpha$, $\rho_x, \rho_u$ be as in the Definition \ref{DefLISS_LF}. Take an arbitrary control $u \in U_c$ with $\|u\|_{U_c} \leq \rho_u$ such that 
\[
I=\{x \in D:\|x\|_X \leq \rho_x,\ V(x) \leq \psi_2 \circ \chi(\|u\|_{U_c}) \leq \rho_x \} \subset int(D).
\]
Such $u$ exists, because $0 \in int(D)$. 

Firstly we prove, that $I$ is invariant w.r.t. $\Sigma$, that is: $\forall x \in I \Rightarrow x(t)=\phi(t,x,u) \in I$, $t \geq 0$. 

If $u \equiv 0$, then $I=\{0\}$, and $I$ is invariant, because $x=0$ is the equilibrium point of $\Sigma$. Consider $u \not \equiv 0$.

If $I$ is not invariant w.r.t. $\Sigma$, then, due to continuity of $\phi$ w.r.t. $t$ (continuity axiom of $\Sigma$), $\exists t_*>0$, such that $V(x(t_*))=\psi_2 \circ \chi(\|u\|_{U_c})$, and therefore $\|x(t_*)\|_X \geq \chi(\|u\|_{U_c})$. 

The input to the system $\Sigma$ after time $t^*$ is $\tilde{u}$, defined by  $\tilde{u}(\tau)=u(\tau + t^*)$, $\tau \geq 0$. According to the assumptions of the theorem $\|\tilde{u}\|_{U_c} \leq \|u\|_{U_c}$.
Then from \eqref{GainImplikation} it follows, that $\dot{V}_{\tilde{u}}(x(t_*))=-\alpha(\|x(t_*)\|_X)<0$. Thus, the trajectory cannot escape the set $I$.

Now take arbitrary $x_0$: $\|x_0\|_X \leq \rho_x$. As long as $x_0 \not\in I$, we have the following differential inequality ($x(t)$ is the trajectory, corresponding to the initial condition $x_0$):
\[
\dot{V}(x(t)) \leq - \alpha(\|x(t)\|_X) \leq - \alpha \circ \psi_2^{-1}(V(x(t))).
\]
From the comparison principle (see \cite{LSW96}, Lemma 4.4 for $y(t)=V(x(t))$) it follows, that $\exists\ \tilde{\beta} \in \KL:\ V(x(t)) \leq \tilde{\beta}(V(x_0),t)$, and consequently:
\begin{equation}
\label{BetaSchaetzung}
\|x(t)\|_X \leq \beta(\|x_0\|_X,t), \forall t:\ x(t) \notin I,
\end{equation}
where $\beta(r,t)=\psi^{-1}_1 \circ \tilde{\beta}(\psi_2^{-1}(r),t)$, $\forall r,t \geq 0$.\\
From the properties of $\KL$ functions it follows, that $\exists t_1$:
\[
t_1:= \inf_{t \geq 0}\{ x(t)=\phi(t,x_0,u) \in I\}.
\]
From the invariance of the set $I$ we conclude, that 
\begin{equation}
\label{GammaSchaetzung}
\|x(t)\|_X \leq \gamma(\|u\|_{U_c}), \ t>t_1,
\end{equation}
where $\gamma=\psi_1^{-1} \circ \psi_2 \circ \chi \in \K$.

Our estimates hold for arbitrary control $u$: $\|u\|_{U_c} \leq \rho_u$, thus, combining \eqref{BetaSchaetzung} and \eqref{GammaSchaetzung}, we obtain the claim of the theorem. 

To prove, that from existence of ISS-Lyapunov function it follows ISS of $\Sigma$, one has to argue as above but with $\rho_x=\rho_u=\infty$. \hfill $\blacksquare$
\end{proof}


\begin{remark}
The assumption in the Theorem~\ref{LyapunovTheorem} concerning the properties of $U_c$ holds for many usual function classes, such as $PC(\R_+,U)$, $L_p (\R_+,U)$, $p \geq 1$, $L_{\infty}(\R_+,U)$, Sobolev spaces etc.
\end{remark}

We are going to prove, that our definition of an ISS-Lyapunov function, applied for an ODE system, is resolved to the standard definition of an ISS-Lyapunov function \cite{SoW95}.

Firstly we reformulate the definition of LISS-LF for the case, when $U_c= PC(\R_+,U)$.

\begin{proposition}
\label{LISS_LF_PC}
A continuous function $V:D \to \R_+$, $D \subset X$, $0 \in int(D)=D \backslash \partial D$ is a LISS-Lyapunov function for $\Sigma = (X,PC(\R_+,U),\phi)$ if and only if there exist $\rho_x,\rho_u>0$ and functions $\psi_1,\psi_2 \in \Kinf$, $\tilde{\chi} \in \K$ and positive definite function $\alpha$, such that:
\[
\psi_1(\|x\|_X) \leq V(x) \leq \psi_2(\|x\|_X), \quad \forall x \in D
\]
and 
$\forall x \in X: \|x\|_X \leq \rho_x, \; \forall \xi \in U: \|\xi\|_U \leq \rho_u$ it holds
\begin{equation}
\label{GainImplikation_PC}
 \|x\|_X \geq \tilde{\chi}(\|\xi\|_U) \  \Rightarrow  \ \dot{V}_u(x) \leq -\alpha(\|x\|_X),
\end{equation}
for all $u \in U_c$: $\|u\|_{U_c} \leq \rho_u$ with $u(0)=\xi$.
\end{proposition}

\begin{proof}
We begin with sufficiency. Let $u \in U_c=PC(\R_+,U)$, $\|u\|_{U_c} \leq \rho_u$.
Take arbitrary $x \in X$ and assume that $\|x\|_X \geq \chi(\|u\|_{U_c})$. 
Then $\|x\|_X \geq \chi(\|u(0)\|_{U})$ and according to \eqref{GainImplikation_PC} for this $u$ it holds $\dot{V}_u(x) \leq -\alpha(\|x\|_X)$. The implication \eqref{GainImplikation} is proved and thus $V$ is a LISS-Lyapunov function according to Definition \ref{DefLISS_LF}. 

Let us prove necessity. 
Take arbitrary $u \in U_c$, and for arbitrary $s >0$ consider the input $u_s \in U_c$ defined by
\[
{u}_s(\tau):= 
\left\{ 
\begin{array}{l}
{u(\tau),\tau \in [0,s],} \\
{u(s), \tau >s.}
\end{array}
\right.
\]
Due to causality of $\Sigma$, $\phi(t,x,u)= \phi(t,x,u_s)$ for all $t \in [0,s]$, and according to the definition of Lie derivative we obtain $\dot{V}_u(x) = \dot{V}_{u_s}(x)$.
Let $u \in U_c$ and $\|u\|_{U_c} \leq \rho_u$. 
Then also $\|u_s\|_{U_c} \leq \rho_u$ and since $V$ is a LISS-Lyapunov function it follows that
\[
\|x\|_X \geq \chi(\|u_s\|_{U_c}) \  \Rightarrow  \ \dot{V}_{u_s}(x) \leq -\alpha(\|x\|_X).
\]
Then it holds also
\begin{equation}
\label{HilfImpl_PC}
 \|x\|_X \geq \chi(\|u_s\|_{U_c}) \  \Rightarrow  \ \dot{V}_u(x) \leq -\alpha(\|x\|_X).
\end{equation}
Since $U_c= PC(\R_+,U)$, it follows that for arbitrary $u \in U_c$ and arbitrary $\eps>0$ there exists $\tau >0$ such that $\|u_{\tau}\|_{U_c} \leq (1+\eps)\|u(0)\|_U$.
Then from \eqref{HilfImpl_PC} it follows that 
\begin{equation*}
\|x\|_X \geq \tilde{\chi}(\|u(0)\|_{U}) \  \Rightarrow  \ \dot{V}_u(x) \leq -\alpha(\|x\|_X),
\end{equation*}
where $\tilde{\chi}(r)=\chi( (1+ \eps) r)$, for all $r \geq 0$.

Since $u \in U_c$, $\|u\|_{U_c} \leq \rho_u$ has been chosen arbitrarily, the necessity is proved. \hfill $\blacksquare$
\end{proof}

Now consider an ODE system
\begin{equation}
\label{ODE_System}
\dot{x}(t)=f(x(t),u(t)), \ x(t) \in \R^n, \ u(t) \in \R^m.
\end{equation}
System \eqref{ODE_System} defines a time-invariant control system $\Sigma=(X,U_c,\phi)$, where $X = \R^n$, $U_c=L_{\infty}(\R_+,\R^m)$ and $\phi(t,x_0,u)$ is a solution of \eqref{ODE_System} subject to a given input $u \in L_{\infty}(\R_+,\R^m)$ and initial condition $x(0)=x_0$.

Let $V:D \to \R_+$, $D \subset \R^n$, $0 \in int(D)=D \backslash \partial D$ be locally Lipschitz continuous function (and thus it is differentiable almost everywhere by Rademacher's theorem).
For such systems $\dot{V}_u(x)$ can be computed for almost all $x$ and the implication 
\eqref{GainImplikation_PC} resolves to 
\begin{equation*}
 \|x\|_X \geq \chi(\| \xi \|_U) \  \Rightarrow  \ \nabla V \cdot f(x,\xi) \leq -\alpha(\|x\|_X).
\end{equation*}
Using this implication instead of \eqref{GainImplikation_PC}, we obtain the standard definition of LISS-Lyapunov function for finite-dimensional systems.
Thus, Definition \ref{DefLISS_LF} is consistent with the existing definitions of LISS-Lyapunov functions for ODE systems.

Note that the system \eqref{ODE_System} is time-invariant, for the space $L_{\infty}(\R_+, \R^m)$ the assumption of the Theorem \ref{LyapunovTheorem} holds, and we obtain basic result from finite-dimensional theory that existence of an (L)ISS-Lyapunov function implies its (L)ISS.


In the following subsection we will need certain type of a density argument, which we state here without a proof.

Let $\Sigma:=(X,U_c,\phi)$ be a control system. Let $\hat{X}$, $\hat{U}_c$ be some dense normed linear subspaces of $X$ and $U_c$ respectively, and let $\hat{\Sigma}:=(\hat{X},\hat{U}_c,\phi)$ be the system, generated by the same as in $\Sigma$ transition map $\phi$, but restricted to the state space $\hat{X}$ and space of admissible inputs $\hat{U}_c$.

Assume that $\phi$ depends continuously on inputs and on initial states, that is
$\forall x \in X, \forall u \in U_c, \forall T>0$ and $\forall \eps>0$ there exist $\delta>0$, such that $\forall x' \in X: \|x-x'\|_X< \delta$ and 
$\forall u' \in U_c: \|u-u'\|_{U_c}< \delta$ it holds
\[
\|\phi(t,x,u)-\phi(t,x',u')\|_X< \eps, \quad \forall t \in [0,T].
\]

Now we have the following result
\begin{Lemm}
\label{DensityArg}
Let $\hat{\Sigma}$ be ISS. Then $\Sigma$ is also ISS with the same $\beta$ and $\gamma$
in the estimate \eqref{iss_sum}.
\end{Lemm}

In the next subsection we demonstrate an application of the theory developed in this section on an example from parabolic PDEs.




\subsection {Example}

Consider the following system
\begin{equation}
\label{UnserSystem}
\left
\{
\begin {array} {l} 
{\frac{\partial s}{\partial t} = \frac{\partial^2 s}{\partial x^2} - f(s) + u^m(x,t), \quad x \in (0,\pi),\ t>0,}\\
{s(0,t) = s(\pi,t)=0.}
\end {array}
\right.	
\end{equation}

We assume, that $f$ is locally Lipschitz continuous, monotonically increasing up to infinity, $f(-r)=-f(r)$ for all $r \in \R$ (in particular, $f(0)=0$), and $m \in (0,1]$.


To reformulate \eqref{UnserSystem} as an abstract differential equation we define operator $A$ by $As:= \frac{d^2s }{dx^2}$ with $D(A)=H^1_0(0,\pi) \cap H^2(0,\pi)$. 

The norm on $H^1_0(0,\pi)$ we choose as $\|s\|_{H^1_0(0,\pi)}:= \left( \int_0^\pi{   \left(\frac{\partial s}{\partial x} \right)^2 dx} \right)^\frac{1}{2}.$

Operator $A$ generates an analytic semigroup on $L_2(0,\pi)$. System \eqref{UnserSystem} takes the form
\begin{equation}
\label{UnserSystem_AbstrForm}
{\frac{\partial s}{\partial t} = As - F(s) + u^m, \ t>0,}
\end{equation}
where $F$ is defined by $F(s(t)) (x):=f(s(x,t))$, $x \in (0,\pi)$.

Equation \eqref{UnserSystem_AbstrForm} defines a control system with the state space $X = H^1_0(0,\pi)$ and input function space $U_c=C(\R_+,L_2(0,\pi))$. 

Consider the following ISS-Lyapunov function candidate:
\begin{equation}
\label{LF}
V(s):= \int_0^\pi{ \left( \frac{1}{2} \left(\frac{\partial s}{\partial x} \right)^2+ \int_0^{s(x)}{f(y)dy} \right) dx}.
\end{equation}

We are going to prove, that $V$ is an ISS-Lyapunov function.

Under the above assumptions about function $f$ it holds that $\int_0^{r}{f(y)dy} \geq 0$ for every $r \in \R$.

We have to verify the estimates \eqref{LyapFunk_1Eig} for a function $V$. 
The estimate from below is easy:
\begin{equation}
\label{LF_1condition}
V(s) \geq \int_0^\pi{\frac{1}{2} \left(\frac{\partial s}{\partial x} \right)^2 dx} = \frac{1}{2}  \|s\|^2_{H^1_0(0,\pi)}.
\end{equation}

Let us find an estimate from above. We have
\begin{equation*}
V(s)= \int_0^\pi{ \frac{1}{2} \left(\frac{\partial s}{\partial x} \right)^2}\ dx + 
\int_0^\pi{ \int_0^{s(x)}{f(y)dy} \ dx}.
\end{equation*}

According to the embedding theorem for Sobolev spaces (see \cite[Theorem 6, p. 270]{Eva98}), every $s\in H^1_0(0,\pi)$ belongs actually to $C^{\frac{1}{2}}(0,\pi)$ (H\"older space with H\"older exponent $\frac{1}{2}$). Moreover, there exists a constant $C$, which does not depend on $s\in H^1_0(0,\pi)$, such that
\begin{equation}
\label{EmbeddingIneq}
\|s\|_{C^{\frac{1}{2}}(0,\pi)} \leq C \|s\|_{H^1_0(0,\pi)},\; \forall s \in H^1_0(0,\pi).
\end{equation}
Define $\psi:\R_+ \to \R_+$ by $\psi(r):=\frac{1}{2} r^2 + \sup_{s:\ \|s\|_{H^1_0(0,\pi)} \leq r} \int_0^\pi \int_0^{s(x)}{f(y)dy} dx$. \\
Inequality \eqref{EmbeddingIneq} and the fact that $\|s\|_{C(0,\pi)} \leq \|s\|_{C^{\frac{1}{2}}(0,\pi)}$ for all $s \in C^{\frac{1}{2}}(0,\pi)$ imply
\begin{eqnarray}
\psi(r)&=&\frac{1}{2} r^2 + \sup_{s:\ C\|s\|_{H^1_0(0,\pi)} \leq Cr} \int_0^\pi \int_0^{s(x)}{f(y)dy} dx \\
      &\leq& \frac{1}{2} r^2 + \sup_{s:\ \|s\|_{C(0,\pi)} \leq Cr} \int_0^\pi \int_0^{s(x)}{f(y)dy} dx \leq \psi_2(r),
\end{eqnarray}
where $\psi_2(r) := \frac{1}{2} r^2 + \pi \int_0^{Cr}{f(y)dy}$.
Since $f$, restricted to positive values of the argument, belongs to $\Kinf$,  $\psi_2$ is also $\Kinf$-function.

Finally, for all $s \in H^1_0(0,\pi)$ we have: 
\begin{eqnarray}
\label{1Eig_Beispiel}
\frac{1}{2}  \|s\|^2_{H^1_0(0,\pi)} \leq V(s) \leq \psi_2(\|s\|_{H^1_0(0,\pi)}),
\end{eqnarray}
and the property \eqref{LyapFunk_1Eig} is verified.

Let us compute the Lie derivative of $V$
\begin{eqnarray*}
\dot{V}(s) &=& \int_0^\pi{ \frac{\partial s}{\partial x} \frac{\partial^2 s}{\partial x \partial t}+ f(s(x))  \frac{\partial s}{\partial t} dx} \\
           &=& \left[\frac{\partial s}{\partial x} \frac{\partial s}{\partial t} \right]_{x=0}^{x=\pi} +\int_0^\pi{ \left( -  \frac{\partial^2 s}{\partial x^2}   \frac{\partial s}{\partial t}+ f(s(x)) \frac{\partial s}{\partial t}  \right)dx}.
\end{eqnarray*}

From boundary conditions it follows ${ \frac{\partial s}{\partial t}(0,t) = \frac{\partial s}{\partial t}(\pi,t)=0}$. Thus, substituting expression for $\frac{\partial s}{\partial t}$, we obtain
\[
\dot{V}(s)= - \int_0^\pi{ \left(\frac{\partial^2 s}{\partial x^2}- f(s(x)) \right)^2 dx} + \int_0^\pi{ \left(\frac{\partial^2 s}{\partial x^2}- f(s(x)) \right)(-u^m) dx}.
\]
Define 
\[
I(s):=\int_0^\pi{ \left(\frac{\partial^2 s}{\partial x^2}- f(s(x)) \right)^2 dx}.
\]
Using the Cauchy-Schwarz inequality for the second term, we have:
\begin{equation}
\label{LeiAb_V_IS}
\dot{V}(s) \leq - I(s) +  \sqrt{I(s)} \ \|u^m\|_{L_2(0,\pi)}.
\end{equation}

Now let us consider $I(s)$
\begin{eqnarray*}
I(s) &=&\int_0^\pi{ \left(\frac{\partial^2 s}{\partial x^2} \right)^2 dx} -2\int_0^\pi{ \frac{\partial^2 s}{\partial x^2} f(s(x)) dx}+ \int_0^\pi{f^2(s(x)) dx} \\
     &=& \int_0^\pi {\left(\frac{\partial^2 s}{\partial x^2} \right)^2 dx} +2\int_0^\pi{ \left(\frac{\partial s}{\partial x} \right)^2 \frac{\partial f}{\partial s}(s(x)) dx}
   + \int_0^\pi{f^2(s(x)) dx}  \\
	  &\ge& \int_0^\pi{ \left(\frac{\partial^2 s}{\partial x^2} \right)^2 dx}.
\end{eqnarray*}

For $s \in H^1_0(0,\pi) \cap H^2(0,\pi)$ it holds (see \cite{Hen81}, p. 85), that
\[
\int_0^\pi{ \left(\frac{\partial^2 s}{\partial x^2} \right)^2 dx} \ge \int_0^\pi{ \left(\frac{\partial s}{\partial x} \right)^2 dx}.
\]

Overall, we have:
\begin{equation}
\label{IS_Abschaetzung}
I(s) \geq \|s\|^2_{H^1_0(0,\pi)}.
\end{equation}

Let us consider $\|u^m\|_{L_2(0,\pi)}$. Using the H\"older inequality, we obtain:
\begin{eqnarray}
\label{NormAbsch2}
\|u^ {m }\|_{L_2(0,\pi)} &=& \left( \int_0^{\pi} {u^ {2m } \cdot 1 \ dx} \right)^{\frac{1}{2}} \notag \\ 
   &\leq&
{\left( \int_0^{\pi} {u^2\ dx} \right)^{ \frac{m }{2}}  \left( \int_0^{\pi} { 1^\frac{1}{{1- m }} dx} \right)^{\frac{1- m }{2}} = \pi^{\frac{1- m }{2}}   \|u\|^{m}_{L_2(0,\pi)}}.
\end{eqnarray}

Now we choose the gain as 
\[
\chi(r)=a \pi^{\frac{1- m}{2}} r^m,\; a>1. 
\]
If $\chi(\|u\|_{L_2(0,\pi)}) \le \|s\|_{H^1_0(0,\pi)}$, we obtain from \eqref{LeiAb_V_IS}, using \eqref{NormAbsch2} and \eqref{IS_Abschaetzung}:
\begin{equation}
\label{LyapGainAbsch}
\dot{V}(s) \leq - I(s) +  \frac{1}{a} \sqrt{I(s)} \|s\|_{H^1_0(0,\pi)} \leq (\frac{1}{a} -  1)I(s) \leq (\frac{1}{a} -  1) \|s\|^2_{H^1_0(0,\pi)}.
\end{equation}

The above computations are valid for states $s \in \hat{X}$:  $\hat{X}:=\{s \in C^{\infty}([0,\pi]):\; s(0)=s(\pi)=0 \}$ and inputs $u \in \hat{U_c}$, $\hat{U_c}:=C(\R_+,C^{\infty}([0,\pi]))$.

The system $(\hat{X},\hat{U_c},\phi)$, where $\phi(\cdot,s,u)$ is a solution of \eqref{UnserSystem}for $s \in \hat{X}$ and $u \in \hat{U_c}$, possesses the ISS-Lyapunov function and consequently is ISS according to Proposition~\ref{LISS_LF_PC}.

It is known, that $\hat{X}$ is dense in $H^1_0(0,\pi)$ and $\hat{U_c}$ is dense in $C(\R_+,L_2([0,\pi]))$. According to the Lemma~\ref{DensityArg} the system \eqref{UnserSystem} is also ISS (with $X=H^1_0(0,\pi)$, $U_c=C(\R_+,L_2(0,\pi))$).

%



\begin{remark}
\label{RemUX}
In the example we have taken  $U=L_2(0,\pi)$ and $X=H^1_0(0,\pi)$. But in case of interconnection with other parabolic systems (when we identify input $u$ with the state of the other system), that have state space $H^1_0(0,\pi)$ (as our system), we have to choose $U=X=H^1_0(0,\pi)$. In this case we can continue the estimates \eqref{NormAbsch2}, using Friedrichs' inequality 
\[
\int_0^\pi{ s^2(x) dx} \le \int_0^\pi{ \left(\frac{\partial s}{\partial x} \right)^2 dx}
\]
 to obtain
\begin{equation}
\label{NormAbsch4}
\|u^m\|_{L_2(0,\pi)} \leq \pi^{\frac{1- m}{2}} \|u\|^{m}_{H^1_0(0,\pi)}
\end{equation}
and choosing the same gains, prove the input-to state stability of \eqref{UnserSystem_AbstrForm} w.r.t. spaces $X=H^1_0(0,\pi)$, $U_c=C(\R_+,H^1_0(0,\pi))$.
\end{remark}

\begin{remark}
The input-to-state stability for semilinear parabolic PDEs has been studied also in the recent paper \cite{MaP11}. However, the definition of ISS and of ISS-Lyapunov function in that paper are different from used in our paper.
In particular, consider the property of \eqref{LyapFunk_1Eig} of ISS-Lyapunov function. The corresponding property (2) from \cite{MaP11} is not equivalent to \eqref{LyapFunk_1Eig} for 
$X:=C^2([0,L],\R^n)$ equipped with the $L_2$-norm (which is chosen as the state space in \cite{MaP11}), since the expression in (2) from \cite{MaP11} cannot be bounded by a function of $L_2$-norm of an element of $X$ in general.
\end{remark}


\section{Linearization}
\label{Linearisierung}

In this section we prove two theorems, stating that a nonlinear system is LISS provided its linearization is ISS. One of them needs less restrictive assumptions, but it doesn't provide us with a LISS-Lyapunov function for the nonlinear system. In the other theorem it is assumed, that the state space is a Hilbert space. This assumption yields a form of LISS-Lyapunov function.

Consider the system
\begin{equation}
\label{InfiniteDim}
\dot{x}(t)=Ax(t)+f(x(t),u(t)), \ x(t) \in X, u(t) \in U,
\end{equation}
where $X$ is a Banach space, $A$ is the generator of a $C_0$-semigroup, $f:X \times U \to X$ is defined on some open set $Q$, $(0,0)$ is in interior of $Q$ and
$f(0,0)=0$, thus $x \equiv 0$ is an equilibrium point of \eqref{InfiniteDim}.


In this section we assume, that $f$ can be decomposed as
\[
f(x,u)=Bx+ Cu + g(x,u),
\]
where $B \in L(X)$, $C \in L(U,X)$ and
for each constant $w>0$ there exist $\rho>0$, such that $\forall x: \|x\|_X \leq \rho,\ \forall u: \|u\|_U \leq \rho$ it holds
\begin{equation}
\label{G_Schranke}
\|g(x,u)\|_X \leq w (\|x\|_X+\|u\|_U).
\end{equation}

Consider also the linear approximation of a system \eqref{InfiniteDim}, given by
\begin{equation}
\label{LinearInfiniteSyst}
\dot{x}=Rx + Cu,
\end{equation}
where $R=A+B$ is an infinitesimal generator of a $C_0$-semigroup (which we denote by $T$), as the sum of a generator $A$ and bounded operator $B$.

Our first result of this section is
\begin{Satz}
\label{LinearisationSatz_1}
If \eqref{LinearInfiniteSyst} is ISS, then \eqref{InfiniteDim} is LISS.
\end{Satz}

\begin{proof}
System \eqref{LinearInfiniteSyst} is ISS, then according to Proposition~\ref{HauptLinSatz} and Lemma~\ref{HilfsAequivalenzen} the semigroup $T$ is exponentially stable, that is for some $K,h>0$ it holds $\|T(t)\| \leq Ke^{-ht}$.

For a trajectory $x(\cdot)$ it holds
\begin{equation*}
x(t)=T(t)x_0 + \int_0^t T(t-s) \left(Cu(s)+g(x(s),u(s))\right) ds.
\end{equation*}
We have:
\begin{equation*}
\|x(t)\|_X \leq Ke^{-ht} \|x_0\|_X + K \int_0^t e^{-h(t-s)} (\|C\| \|u(s)\|_U+\|g(x(s),u(s))\|_X) ds.
\end{equation*}
Take small enough $w>0$. Then there exists some $r>0$, such that \eqref{G_Schranke} holds for all $x,u$: $\|x\|_X \leq r$ and $\|u\|_U \leq r$.
Take initial condition $x_0$ and input $u$ such that $\|u\|_{U_c} < r$ and $\|x_0\|_X < r$. Then, due to continuity of the trajectory, there exist some $t^*>0$ such that $\|x(t)\|_X < r$, $t \in [0,t^*]$. 

For all $t \in [0,t^*]$ and every $\eps<h$ using "fading-memory" estimates (see, e.g. \cite{KaJ11b}) 
we obtain
\begin{eqnarray*} 
\|x(t)\|_X  & \leq &Ke^{-ht} \|x_0\|_X \\
   & &  + K \int_0^t e^{-\eps(t-s)} e^{-(h-\eps)(t-s)} (\|C\| \|u(s)\|_U+w (\|x(s)\|_X+\|u(s)\|_U) ) ds \\
      & \leq & Ke^{-ht} \|x_0\|_X + \frac{K}{\eps} \sup\limits_{0\leq s\leq t} {e^{-(h-\eps)(t-s)} ( (\|C\|+w) \|u(s)\|_U+w \|x(s)\|_X )}.			
\end{eqnarray*}

Define $\psi$ and $v$ by $\psi(t):=e^{(h-\eps)t}x(t)$ and $v(t):=e^{(h-\eps)t}u(t)$ respectively. Multiplying the previous inequality by $e^{(h-\eps)t}$, we obtain:
\begin{eqnarray*}
\|\psi(t)\|_X   & \leq & Ke^{- \eps t} \|x_0\|_X + \frac{K}{\eps} (\|C\|+w) \sup\limits_{0\leq s\leq t} \|v(s)\|_U + \frac{K}{\eps} w \sup\limits_{0\leq s\leq t} \|\psi(s)\|_X.
\end{eqnarray*}

Assume that $w$ is so that $1-\frac{K}{\eps} w > 0$. Taking supremum from the both sides, we obtain:
\begin{eqnarray*}
\sup\limits_{0\leq s\leq t} \|\psi(s)\|_X   & \leq & \frac{1}{1-\frac{K}{\eps}w} \left(  K \|x_0\|_X + \frac{K}{\eps} (\|C\|+w) \sup\limits_{0\leq s\leq t} \|v(s)\|_U \right).
\end{eqnarray*}
In particular,
\begin{eqnarray*}
\|\psi(t)\|_X   & \leq & \frac{1}{1-\frac{K}{\eps}w} \left(  K \|x_0\|_X + \frac{K}{\eps} (\|C\|+w) \sup\limits_{0\leq s\leq t} \|v(s)\|_U \right).
\end{eqnarray*}
Returning to the variables $x,u$, we have:
\begin{eqnarray*}
\|x(t)\|_X  \leq \frac{K}{1-\frac{K}{\eps}w} \left(e^{-(h-\eps) t} \|x_0\|_X + \frac{(\|C\|+w)}{\eps}  \sup\limits_{0\leq s\leq t} e^{-(h-\eps) (t-s)} \|u(s)\|_U \right).
\end{eqnarray*}
Taking $\|u\|_{U_c}$ and $\|x_0\|_X$ small enough we guarantee that $\|x(t)\|_X < r$ for all $t \in [0,t^*]$. Because of BIC property it is clear, that $t^*$ can be chosen arbitrarily large. 
Thus, the last estimate proves LISS of the system \eqref{InfiniteDim}.
\hfill $\blacksquare$
\end{proof}

\begin{remark}
In the proof of the previous theorem the last inequality is a "fading memory" estimate of a norm of a state. This shows, that the system is not only ISS, but also ISDS, see \cite{Gru02a}.
\end{remark}

\subsection{Constructions of LISS-Lyapunov functions}

In this subsection we are going to use linearization in order to construct LISS-Lyapunov functions for nonlinear systems.

In addition to assumptions in the beginning of the Section \ref{Linearisierung} suppose that $X$ is a Hilbert space with a scalar product $\lel \cdot, \cdot \rir$, and $A$ generates an analytic semigroup on $X$. 

Recall that a self-adjoint operator $P \in L(X)$ is positive if $\lel P x,x \rir > 0$ for all $x \in X$, $x \neq 0$. A positive operator $P$ is called coercive, if $\exists \epsilon >0$, such that
\[
\lel Px,x \rir \geq \epsilon \|x\|^2_X \quad \forall x \in D(P).
\]

%

Since operator $A$ is an infinitesimal generator of an analytic semigroup and $B$ is bounded, 
$R=A+B$ also generates an analytic semigroup.

Let system \eqref{LinearInfiniteSyst} be ISS. Then, according to Proposition \ref{HauptLinSatz},  \eqref{LinearInfiniteSyst} is exponentially 0-UGAS$x$. By Lemma \ref{HilfsAequivalenzen} this implies that $R$ generates exponentially stable semigroup.
By \cite[Theorem 5.1.3, p. 217]{CuZ95} this is equivalent to the existence of 
a positive bounded operator $P\in L(X)$, for which it holds that
\begin{equation}
\label{OperatorBedingung}
\lel Rx,Px \rir + \lel Px,Rx \rir=-\|x\|^2_X, \quad \forall x \in D(R).
\end{equation}

%

If an operator $P$ is coercive, then a LISS-Lyapunov function for a system \eqref{InfiniteDim} can be constructed. More precisely, it holds
\begin{Satz}
\label{LinearisationSatz}
If the system \eqref{LinearInfiniteSyst} is ISS, and there exists a coercive operator $P$, satisfying \eqref{OperatorBedingung}, then a
LISS-Lyapunov function of \eqref{InfiniteDim} can be constructed in the form
\begin{equation}
\label{LF_LinearInfiniteSyst}
V(x)=\lel Px,x \rir.
\end{equation}
\end{Satz}

\begin{proof}
%
%

Since $P$ is bounded and coercive, for some $\epsilon>0$ it holds 
\[
\epsilon \|x\|^2_X \leq \lel Px,x \rir \leq \|P\| \|x\|^2_X, \quad \forall x \in X,
\]
and estimate \eqref{LyapFunk_1Eig} is verified.

Let us compute the Lie derivative of $V$ w.r.t. the system \eqref{InfiniteDim}. 
Firstly consider the case, when $x \in D(R)=D(A)$. We have
\begin{eqnarray*}
\dot{V}(x) &=& \lel P\dot{x},x \rir +  \lel Px,\dot{x} \rir \\
 &=&\lel P(Rx +Cu+ g(x,u)),x \rir + \lel Px,Rx +Cu+ g(x,u) \rir  \\
&=& \lel P(Rx),x \rir +  \lel Px,Rx \rir + 
\lel P(Cu+g(x,u)),x \rir +  \lel Px,Cu+g(x,u) \rir.
\end{eqnarray*}
We continue estimates using the property 
\[
\lel P(Rx),x \rir = \lel Rx,Px \rir,
\]
which holds for positive operators, equality \eqref{OperatorBedingung} and for the last two terms Cauchy-Schwarz inequality in the space $X$
\begin{eqnarray*}
\dot{V}(x) &\leq& -\|x\|^2_X + \|P(Cu+g(x,u))\|_X  \|x\|_X + \|Px\|_X \|Cu+g(x,u)\|_X \\ &\leq& -\|x\|^2_X + \|P\| \|(Cu+g(x,u))\|_X  \|x\|_X + \|P\| \|x\|_X \|Cu+g(x,u)\|_X \\ &\leq& -\|x\|^2_X + 2 \|P\| \|x\|_X (\|C\| \|u\|_U + \|g(x,u)\|_X).
\end{eqnarray*}
For each $w>0$ $\exists \rho$, such that $\forall x: \|x\|_X \leq \rho,\ \forall u: \|u\|_U \leq \rho$ it holds \eqref{G_Schranke}.
Using \eqref{G_Schranke} we continue above estimates
\[
\dot{V}(x) \leq -\|x\|^2_X + 2 w \|P\| \|x\|^2_X  + 2 \|P\| (\|C\| +w)\|x\|_X \|u\|_U.
\]
Take $\chi(r):=\sqrt{r}$. Then for $\|u\|_U \leq \chi^{-1}(\|x\|_X)=\|x\|^2_X$ we have:
\begin{equation}
\label{LyapEstimate}
\dot{V}(x) \leq -\|x\|^2_X + 2 w \|P\| \|x\|^2_X  + 2 \|P\| (\|C\| +w)\|x\|^3_X. 
\end{equation}
Choosing $w$ and $\rho$ small enough the right hand side can be estimated from above by some negative quadratic function of $\|x\|_X$. 

These derivations hold for $x \in D(R) \subset X$. If $x \notin D(R)$, then for all admissible $u$ the solution $x(t) \in D(R)$ and $t \to V(x(t))$ is a continuously differentiable function for all $t >0$ (these properties follow from the properties of solutions $x(t)$, see Theorem 3.3.3 in \cite{Hen81}). 

Therefore, by the mean-value theorem, $\forall t>0$ $\exists t_* \in (0,t)$
\[
\frac{1}{t}(V(x(t))-V(x))=\dot{V}(x(t_*)).
\]
Taking the limit when $t \to +0$ we obtain that \eqref{LyapEstimate} holds for all $x \in X$.

This proves that $V$ is a LISS-Lyapunov function with $\|x\|_X \leq \rho$, $\|u\|_U \leq \rho$ and consequently \eqref{InfiniteDim} is LISS.
\hfill $\blacksquare$
\end{proof}

\section{Interconnections of input-to-state stable systems}
\label{GekoppelteISS_Systeme}

In this section we study input-to-state stability of an interconnection
of $n$ ISS systems and provide a generalization of Lyapunov small-gain
theorem from \cite{DRW06b} for the case of infinite-dimensional systems.

Consider the interconnected systems of the following form 
\begin{equation}
\label{Kopplung_N_Systeme}
\left\{ 
\begin{array}{l}
\dot{x}_{i}=A_{i}x_{i}+f_{i}(x_{1},\ldots,x_{n},u),\,\, x_{i}(t)\in X_{i}, \ u(t) \in U\\
i=1,\ldots,n,
\end{array}
\right.  
\end{equation}
where the state space of $i$-th subsystem $X_{i}$ is a Banach space and $A_i$ is a generator of $C_{0}$-semigroup on $X_{i}$, $i=1,\ldots, n$. 
The space $U_c$ we take as $U_c=PC(\R_+,U)$ for some Banach space of input values $U$.

The state space of the system \eqref{Kopplung_N_Systeme} we denote by $X=X_{1}\times\ldots\times X_{n}$, 
which is Banach with the norm $\|\cdot\|_{X}:=\|\cdot\|_{X_{1}}+\ldots+\|\cdot\|_{X_{n}}$.

The input space for the $i$-th subsystem is
$\tilde{X}_i:=X_1 \times \ldots \times X_{i-1} \times X_{i+1} \times \ldots \times X_n \times U$. 
The norm in $\tilde{X}_i$ is given by
\[
\|\cdot\|_{\tilde{X}_i}:=\|\cdot\|_{X_{1}}+\ldots + \|\cdot\|_{X_{i-1}} + \|\cdot\|_{X_{i+1}} + \ldots +\|\cdot\|_{X_{n}} + \|\cdot\|_{U}.
\]
The elements of $\tilde{X}_i$ we denote by $\tilde{x}_i=(x_1,\ldots,x_{i-1},x_{i+1},\ldots,x_n,u) \in \tilde{X}_i$.

The transition map of the $i$-th subsystem we denote by $\phi_i:\R_+ \times X_i \times PC(\R_+, \tilde{X}_i) \to X_i$. Define
\[
x{=}(x_{1}^T,{\ldots},x_{n}^T)^T, \ f(x,u){=}(f_{1}(x,u)^T,{\ldots},f_{n}(x,u)^T)^T,\ A {=} \left(
\begin{array}{cccc}
A_{1} & 0 & \ldots & 0\\
0 & A_{2} & \ldots & 0\\
\vdots & \vdots & \ddots & \vdots\\
0 & 0 & \ldots & A_{n}
\end{array}\right),
\]
where $x_i \in X_i,\ i=1,\ldots, n$.
%
Domain of definition of $A$ is given by $D(A)=D(A_{1})\times\ldots\times D(A_{n})$.
Clearly $A$ is a generator of $C_{0}$-semigroup on $X$. 

We rewrite the system \eqref{Kopplung_N_Systeme} in the vector form: 
\begin{equation} 
\label{KopplungHauptSys}
\dot{x}=Ax+f(x,u).
\end{equation}
Since the inputs are piecewise continuous functions, then according to Proposition \ref{LISS_LF_PC} a function $V_i:X_i \to \R_+$ is an ISS-Lyapunov function for the $i$-th subsystem, if there exist functions $\psi_{i1},\psi_{i2}\in\Kinf$, $\chi \in \K$ and positive definite function $\alpha_{i}$, such that
\[
\psi_{i1}(\|x_{i}\|_{X_{i}})\leq V_{i}(x_{i})\leq\psi_{i2}(\|x_{i}\|_{X_{i}}),\quad\forall x_{i}\in X_{i}
\]
and $\forall x_{i}\in X_{i},\;\forall \tilde{x}_i \in \tilde{X}_i$,
for all $v \in PC(\R_+,\tilde{X}_i)$ with $v(0)=\tilde{x}_i$ it holds the implication
\begin{equation}
\label{GainImplikation_iThSys}
 \|x_i\|_{X_i} \geq \chi(\|\tilde{x}_i\|_{\tilde{X}_i}) \  \Rightarrow  \ \dot{V_{i}}(x_i)\leq-\alpha_{i}(V_{i}(x_{i})),
\end{equation}
 where 
\[
\dot{V}_i(x_{i})=\mathop{\overline{\lim}}\limits _{t\rightarrow+0}\frac{1}{t}(V_i(\phi_{i}(t,x_{i},v)))-V_i(x_{i})).
\]
We are going to rewrite the implication \eqref{GainImplikation_iThSys} in a more suitable form.
We have
\begin{eqnarray*}
\psi^{-1}_{i1}(V_{i}(x_{i})) & \geq & \|x_i\|_{X_i} \geq \chi(\|\tilde{x}_i\|_{\tilde{X}_i}) 
 =  \chi \left( \sum_{j=1,j\neq i}^n \|x_j\|_{{X}_j} + \|u\|_U \right)  \\
&\geq& \frac{1}{n+1} \max\{ \max_{j=1,j \neq i}^{n} \{ \chi(\|x_j\|_{{X}_j}) \}, \chi(\|u\|_U ) \}
\end{eqnarray*}
Therefore if $\|x_i\|_{X_i} \geq \chi(\|\tilde{x}_i\|_{\tilde{X}_i})$ holds, then also 
\[
V_i(x_{i})\geq\max\{ \max_{j=1}^{n}\chi_{ij}(V_{j}(x_{j})),\chi_{i}(\|u\|_{U})\}
\]
holds with 
\[
\chi_{ij}(r):= \psi_{i1}\left( \frac{1}{n+1} \chi( \psi^{-1}_{i2} (r) ) \right), \ 
\chi_{i}(r):= \psi_{i1}\left( \frac{1}{n+1} \chi( r ) \right), \ i \neq j,\ r \geq 0.
\]
And thus if \eqref{GainImplikation_iThSys} holds, then it holds also the implication
\begin{equation}
\label{GainImplikationNSyst}
V_i(x_{i})\geq\max\{ \max_{j=1}^{n}\chi_{ij}(V_{j}(x_{j})),\chi_{i}(\|u\|_{U})\} \ \Rightarrow\ \dot{V_{i}}(x_i)\leq-\alpha_{i}(V_{i}(x_{i})).
\end{equation}

The statement, that if \eqref{GainImplikationNSyst} holds, then so is \eqref{GainImplikation_iThSys} can be checked in the same way.

\begin{remark}
Note that we have used in our derivations the above norm on the space $\tilde{X}_i$. For finite-dimensional $\tilde{X}_i$ such derivations can be made for arbitrary norm in $\tilde{X}_i$ due to equivalence of the norms in a finite-dimensional space. However, for infinite-dimensional systems it is not always true.
\end{remark}

In the following we will use the implication form as in \eqref{GainImplikationNSyst} and assume, that for all $i=1,\ldots,n$ for Lyapunov function $V_i$ of the $i$-th system the gains $\chi_{ij}$, $j=1,\ldots,n$ and $\chi_i$ are given.


Gains $\chi_{ij}$ characterize the interconnection structure of
subsystems. 
Let us introduce the gain operator $\Gamma:\R_{+}^{n}\rightarrow\R_{+}^{n}$
defined by 
\begin{equation}
\label{operator_gamma}
\Gamma(s):=\left(\max_{j=1}^{n}\chi_{1j}(s_{j}),\ldots,\max_{j=1}^{n}\chi_{nj}(s_{j})\right),\ s\in\R_{+}^{n}.
\end{equation}

For arbitrary $x,y \in \R^n$ define the relations "$\geq$" and "$<$" on $\R^n$ by
\[
x \geq y \quad \Iff \quad   x_i \geq y_i, \; \forall i=1,\ldots,n,
\]
\[
x < y \quad \Iff \quad    x_i < y_i,  \forall i=1,\ldots,n.
\]

We recall the notion of $\Omega$-path (see \cite{DRW10,Rue10}), useful for investigation of stability of interconnected systems and for construction of a Lyapunov function of the whole interconnection.
\begin{Def} 
A function $\sigma=(\sigma_{1},\dots,\sigma_{n})^{T}:\R_{+}^{n}\rightarrow\R_{+}^{n}$,
where $\sigma_{i}\in\K_{\infty}$, $i=1,\ldots,n$ is called an \textit{$\Omega$-path},
if it possesses the following properties:
\begin{enumerate}
	\item $\sigma_{i}^{-1}$ is locally Lipschitz continuous on $(0,\infty)$;
	\item for every compact set $P\subset(0,\infty)$ there are finite
constants $0<K_{1}<K_{2}$ such that for all points of differentiability
of $\sigma_{i}^{-1}$ we have 
\begin{align*}
0<K_{1}\leq(\sigma_{i}^{-1})'(r)\leq K_{2},\quad\forall r\in P ;
\end{align*}
\item 
\begin{align}
\label{sigma cond 2}
\Gamma(\sigma(r))<\sigma(r),\ \forall r>0.
\end{align}
\end{enumerate}
\end{Def}

\begin{remark}
Note, that for our purposes \eqref{sigma cond 2} can be weakened to 
\begin{align}
\label{sigma cond 3}
\Gamma(\sigma(r)) \leq \sigma(r),\ \forall r>0.
\end{align}

\end{remark}

If operator $\Gamma$ satisfies the small-gain condition, namely for all $\forall\ s\in\R_{+}^{n}\backslash\left\{ 0\right\}$ it holds
\begin{align}
\label{smallgaincondition}
\Gamma(s)\not\geq s \quad \Iff \quad \exists i:  (\Gamma(s))_i < s_i,
\end{align}
then an $\Omega$-path can be constructed as follows (see \cite{KaJ11}, Proposition 2.7 and Remark 2.8):
\begin{equation}
\label{OmegaPfad}
\sigma(t)=Q(at), \forall t \geq 0,\ \mbox{ for some } a \in int(\R^n_+),
\end{equation}
where $Q:\R^n_+ \to \R^n_+$ is defined by
\[
Q(x)=MAX\{x,\Gamma(x), \Gamma^2(x),\ldots, \Gamma^{n-1}(x)\},
\]
with $\Gamma^n(x)=\Gamma \circ \Gamma^{n-1}(x)$, for all $n \geq 2$.
The function $MAX$ for all $u_i \in \R^n$, $i=1,\ldots,m$ is defined by
\[
z=MAX\{u_1,\ldots,u_m\} \in \R^n,\quad z_i=\max\{u_{1i},\ldots,u_{mi}\}.
\]

Note that $\Omega$-path \eqref{OmegaPfad} is only Lipschitz continuous, but with the help of standard mollification arguments (see, \cite{Gru02b}, Appendix B.2 or \cite{Rue07}, Lemma 1.1.6) it can be made smooth. 

%

Now we can state a theorem, that provides sufficient conditions
for a network, consisting of $n$ ISS subsystems to be ISS. 

\begin{Satz}
\label{thm:main1} 
Let for each subsystem of \eqref{Kopplung_N_Systeme}
$V_{i}$ be the ISS-Lyapunov function with corresponding gains $\chi_{ij}$.
If the corresponding operator $\Gamma$ defined by \eqref{operator_gamma}
satisfies the small-gain condition \eqref{smallgaincondition}, then
the whole system \eqref{KopplungHauptSys} is ISS and possesses ISS-Lyapunov
function defined by 
\begin{align}
\label{NeuLyapFun}
V(x):=\max_{i}\{\sigma_{i}^{-1}(V_{i}(x_{i}))\},
\end{align}
 where $\sigma=(\sigma_{1},\ldots,\sigma_{n})^{T}$ is an $\Omega$-path. The Lyapunov gain of the whole system is 
\[
\chi(r):=\max_{i}\sigma_{i}^{-1}(\chi_{i}(r)).
\]
\end{Satz}

For the proof we use the following standard fact from analysis
\begin{Lemm}
\label{LIM_MAX_COR}
Let $f_i:\R \to \R$ are defined and bounded in some neighborhood $D$ of $t=0$. Then it holds
\begin{equation}
\label{BeideRichtungen_Cor}
\overline{\lim_{t \to 0}} \max_{1 \leq i \leq m} \{f_i(t) \} = 
 \max_{1 \leq i \leq m} \{\overline{\lim_{t \to 0}} f_i(t) \}
\end{equation}
\end{Lemm}

The idea of the proof is taken from \cite{DRW06b}.

\begin{proof}
In order to prove that $V$ is a Lyapunov function it is suitable to divide its domain of definition into subsets on which $V$ takes a simpler form. Thus, for all $i\in \{1,\ldots,n\}$ define the set
\[
M_{i}=\left\{ x\in X:\sigma_{i}^{-1}(V_{i}(x_{i})) > \sigma_{j}^{-1}(V_{j}(x_{j})),\,\,\forall j=1,\ldots,n,\ j\neq i\right\}.
\]
From the continuity of $V_i$ and $\sigma_{i}^{-1}$, $i=1,\ldots,n$ it follows that all $M_i$ are open. Also note that $X=\cup_{i=1}^{n}\overline{M}_{i}$ and for all $i \neq j$ holds $M_i \cap M_j=    \emptyset$.
Define 
\[
\gamma(r):=\max_{j=1}^{n}\sigma_{j}^{-1}\circ\gamma_{j}(r).
\]
Take some $i\in \{1,\ldots,n\}$ and pick any $x\in M_{i}$. Assume that $V(x)\geq\gamma(\|\xi\|_U)$ holds. 
Then we obtain
\[
\sigma_{i}^{-1}(V_{i}(x_{i}))=V(x)\geq\gamma(\|\xi\|_U)=\max_{j=1}^{n}\sigma_{j}^{-1}\circ\gamma_{j}(\|\xi\|_U)\geq\sigma_{i}^{-1}(\gamma_{i}(\|\xi\|_U)).
\]
But $\sigma_{i}^{-1} \in \Kinf$, hence it holds
\begin{eqnarray}
\label{eq:GainAbschaetzung}
V_{i}(x_{i}) & \geq & \gamma_{i}(\|\xi\|_U).
\end{eqnarray}
On the other hand, from the condition \eqref{sigma cond 3} we obtain that
\begin{eqnarray*}
V_{i}(x_{i})=\sigma_{i}(V(x))    &\geq&    \max_{j=1}^{n}\chi_{ij}\left(\sigma_{j}\left(V\left(x\right)\right)\right)=\max_{j=1}^{n}\chi_{ij}\left(\sigma_{j}\left(\sigma_{i}^{-1}\left(V_i\left(x_{i}\right)\right)\right)\right) \\
    &>&     \max_{j=1}^{n}\chi_{ij}\left(\sigma_{j}\left(\sigma_{j}^{-1}\left(V_j\left(x_{j}\right)\right)\right)\right)=\max_{j=1}^{n}\chi_{ij}\left(V_j\left(x_{j}\right)\right).
\end{eqnarray*}
Combining it with \eqref{eq:GainAbschaetzung} we obtain
\begin{eqnarray}
\label{eq:LyapImplikation}
V_{i}(x_{i})  \geq  \max\left\{ \max_{j=1}^n \chi_{ij}\left(V_j\left(x_{j}\right)\right),\gamma_{i}(\|\xi\|_U)\right\}.
\end{eqnarray}

Hence condition \eqref{GainImplikationNSyst} implies that for all $x$ the following estimate holds
\begin{eqnarray*}
\frac{d}{dt}V(x) &=& \frac{d}{dt} ( \sigma_{i}^{-1}(V_i(x_i)) )=
\left(\sigma_{i}^{-1}\right)^{\prime}(V_{i}(x_{i}))\frac{d}{dt}V_{i}(x_{i}(t)) \\
&\leq& -\left(\sigma_{i}^{-1}\right)^{\prime}(V_{i}(x_{i}))\alpha_{i}(V_{i}(x_{i}))
=-\left(\sigma_{i}^{-1}\right)^{\prime}(\sigma_{i}(V(x)))\alpha_{i}(\sigma_{i}(V(x))).
\end{eqnarray*}
We set
\[
\alpha(r):=\min_{i=1}^{n} \left\{\left(\sigma_{i}^{-1}\right)^{\prime}(\sigma_{i}(r))\alpha_{i}(\sigma_{i}(r)) \right\}.
\]
Function $\alpha$ is positive definite, because $\sigma_{i}^{-1} \in \Kinf$ and all $\alpha_i$ are positive definite functions.
Overall, for all $x\in \cup_{i=1}^{n}{M}_{i}$ holds
\[
\frac{d}{dt}V(x) \leq -\min_{i=1}^{n} \left(\sigma_{i}^{-1}\right)^{\prime}(\sigma_{i}(V(x)))\alpha_{i}(\sigma_{i}(V(x))) = - \alpha(V(x)).
\]

Now let $x \notin \cup_{i=1}^{n}{M}_{i}$. From $X=\cup_{i=1}^{n}{\overline{M}}_{i}$ it follows that $x \in \cap_{i\in I(x)}\partial {M}_{i}$  for some index set $I(x)\subset\left\{ 1,\ldots,n\right\}$, $|I(x)| \geq 2$.
\[
\cap_{i\in I(x)}\partial {M}_{i} = \{ x \in X: \forall i \in I(x),\ \forall j \notin I(x)\  \sigma_{i}^{-1}(V_{i}(x_{i})) > \sigma_{j}^{-1}(V_{j}(x_{j})),
\]
\[
 \forall i,j \in I(x) \ \sigma_{i}^{-1}(V_{i}(x_{i})) = \sigma_{j}^{-1}(V_{j}(x_{j})) \}. 
\]

Due to continuity of $\phi$ we have, that for all $u \in PC(\R_+,U)$, $u(0)=\xi$ there exists $t^*>0$, such that for all $t \in [0, t^*)$ it follows $\phi(t,x,u) \in \left( \cap_{i\in I(x)} \partial {M}_{i} \right) \cup \left( \cup_{i\in I(x)}{M}_{i} \right)$.



Then, by definition of the derivative we obtain
\begin{eqnarray}
\label{HauptAbschaetzung}
\dot{V}(x) &=& \mathop{\overline{\lim}}\limits _{t\rightarrow+0}\frac{1}{t}(V(\phi(t,x,u)))-V(x)) \\
&=&
\mathop{\overline{\lim}}\limits _{t\rightarrow+0}\frac{1}{t} \left(\max_{i \in I(x)}\{\sigma_{i}^{-1}(V_{i}(\phi_i(t,x,u)))\}-\max_{i \in I(x)}\{\sigma_{i}^{-1}(V_{i}(x_{i}))\} \right)
\end{eqnarray}

From the definition of $I(x)$ it follows that
\[
\sigma_i^{-1}(V_{i}(x_{i}))=\sigma_j^{-1}(V_{j}(x_{j})) \quad \forall i,j\in I(x),
\]
and therefore the index $i$, on which the maximum in $\max_{i \in I(x)}\{\sigma_{i}^{-1}(V_{i}(x_{i}))\}$ is reached, may be always set equal to the index on which the maximum 
$\max_{i \in I(x)}\{\sigma_{i}^{-1}(V_{i}(\phi_i(t,x,u)))\}$ is reached. \\
We continue estimates \eqref{HauptAbschaetzung}
\[
\dot{V}(x)=\mathop{\overline{\lim}}\limits _{t\rightarrow+0}\max_{i \in I(x)}\{\frac{1}{t}\left(\sigma_{i}^{-1}(V_{i}(\phi_i(t,x,u)))-\sigma_{i}^{-1}(V_{i}(x_{i}))\right) \}
\]
Using Lemma \ref{LIM_MAX_COR} we obtain
\[
\dot{V}(x)= \max_{i \in I(x)}\{   \mathop{\overline{\lim}} \limits _{t\rightarrow+0}\frac{1}{t}\left(\sigma_{i}^{-1}(V_{i}(\phi_i(t,x,u)))-\sigma_{i}^{-1}(V_{i}(x_{i}))\right)\}.
\]
Overall, we have that for all $x \in X$ holds
\[
\frac{d}{dt}V(x)=\max_{i}\{\left(\sigma_{i}^{-1}\right)^{\prime}(V_{i}(x_{i}))\frac{d}{dt}V_{i}(x_{i}(t))\}\leq-\alpha(V(x)),
\]
and the theorem is proved for all $x \in X$. \hfill $\blacksquare$
\end{proof}

\begin{remark}
In the recent paper \cite{KaJ11} it was proved a general vector small-gain theorem, that states roughly speaking that if an abstract control system possesses a vector ISS Lyapunov function, then it is ISS. The authors have also shown how from this theorem the small-gain theorems for interconnected systems of ODEs and retarded equations can be derived. It is possible, that the small-gain theorem, similar to the proved in this section, can be derived from the general theorem from \cite{KaJ11}. However,
it seems, that the constructions in \cite{KaJ11} can be provided only for maximum formulation of ISS-Lyapunov functions (as in \eqref{GainImplikationNSyst}). If the subsystems possess ISS-Lyapunov functions in terms of summations, i.e. instead of \eqref{GainImplikationNSyst} one has 
\begin{equation}
\label{GainImplikationNSyst_Sums}
V_i(x_{i})\geq \sum_{j=1}^{n} \chi_{ij}(V_{j}(x_{j})) + \chi_{i}(\|u\|_{U}) \ \Rightarrow\ \dot{V_{i}}(x_i)\leq-\alpha_{i}(V_{i}(x_{i})),
\end{equation}
then it is not clear, how the proofs from \cite{KaJ11} can be adapted for this case. In contrast to it, the counterpart of the Theorem~\ref{thm:main1} in the summation case can be proved with the method, similar to the used in the proof of the Theorem~\ref{thm:main1}, see \cite{DRW10}. However, the small-gain condition will have slightly another form. 
\end{remark}


\subsection{Interconnections of linear systems}

The construction of ISS-Lyapunov function for the interconnections of finite-dimensional input-to-state stable linear systems (see \cite{DRW10}) can be generalized to the case of interconnections of linear systems over Banach spaces.

Let $X_i$, $i=1,\ldots,n$ be Banach spaces. Consider $n$ systems of the form
\begin{align}
\label{StabileSysteme}
\dot{x}_i=A_i x_i(t),\ i=1,\ldots, n,	
\end{align}

where $x_i(t) \in X_i$, $A_i:X_i \to X_i$ is a generator of an analytic semigroup over $X_i$ defined on $D(A_i) \subset X_i$.

Assume that all systems \eqref{StabileSysteme} are {\UGx} and consider the following interconnection
\begin{align}
\label{Kopplung_nSysteme}
\dot{x}_i=A_i x_i(t) + \sum_{j=1}^n {B_{ij}x_j(t)} + C_i u(t),\ i=1,\ldots, n,	
\end{align}
where $B_{ij} \in L(X_j,X_i)$, $i,j \in \{1,\ldots,n\}$ are bounded operators, $u \in U_c=PC(\R_+,U)$ for some Banach space of input values $U$.
We assume, that $B_{ii} = 0$, $i=1,\ldots,n$. Otherwise we can always substitute 
$\tilde{A}_i=A_i + B_{ii}$.

Let us denote $X=X_1 \times \ldots \times X_n$ and introduce the matrix operators $A=diag(A_1,\ldots,A_n):X \to X$, $B=(B_{ij})_{i,j=1,\ldots, n}:X \to X$ and $C=(C_1,\ldots,C_n): U \to X$. Then 
the system \eqref{Kopplung_nSysteme} can be rewritten in the following form
\begin{align}
\label{GrossSystem}
\dot{x}(t)=(A + B)x(t) + Cu(t).
\end{align}

Now we apply Lyapunov technique developed in this section to the system \eqref{Kopplung_nSysteme}.
From Theorem \ref{HauptLinSatz} and Lemma \ref{HilfsAequivalenzen} we have, that $i$-th subsystem of \eqref{Kopplung_nSysteme} is ISS iff the analytic semigroup generated by $A_i$ is exponentially stable. This is equivalent (see 
\cite[Theorem 5.1.3]{CuZ95}) 
to existence of a positive operator $P_i$, for which it holds that
\begin{equation}
\label{OperatorBedingung_ith_Syst}
\lel A_i x_i ,P_i x_i \rir + \lel P_i x_i,A_i x_i \rir \leq -\|x_i\|^2_{X_i}, \quad \forall x_i \in D(A_i).
\end{equation}


Consider a function $V_i$ defined by 
\begin{equation}
\label{LF_ith_Subsyst}
V_i(x_i)=\lel P_i x_i,x_i \rir, x_i \in X_i.
\end{equation}
We assume in what follows that $P_i$ is a coercive operator. 
This implies that
\begin{equation}
\label{ErsteUngl_LF}
a_i^2 \|x_i\|^2_{X_i} \leq V_i(x_i) \leq \|P_i\| \|x_i\|^2_{X_i},
\end{equation}
for some $a_i>0$.

%
%

Differentiating $V_i$ w.r.t. the $i$-th subsystem of \eqref{Kopplung_nSysteme}, we obtain for all $x_i \in D(A_i)$
\[
\dot{V}_i(x_i) = \lel P_i \dot{x}_i,x_i \rir + \lel P_i x_i,\dot{x}_i \rir \leq
\]
\[
(\lel P_i A_i {x}_i,x_i \rir + \lel P_i x_i, A_i x_i \rir ) +
2 \|x_i\|_{X_i} \|P_i\| \left( \sum_{i \neq j} \|B_{ij}\| \|x_j\|_{X_j}  + \|C_i\| \|u\|_U \right). 
\]
Operator $P_i$ is self-adjoint, hence it holds $\lel P_i A_i {x}_i,x_i \rir = \lel A_i {x}_i, P_i x_i \rir$ and by equality \eqref{OperatorBedingung_ith_Syst} we obtain
\[
\dot{V}_i(x_i) \leq  -\|x_i\|^2_{X_i} + 
2 \|x_i\|_{X_i} \|P_i\| \left( \sum_{i \neq j} \|B_{ij}\| \|x_j\|_{X_j}  + \|C_i\| \|u\|_U \right).
\]

Now take $\varepsilon \in (0,1)$ and let
\begin{equation}
\label{Gains_Wahl}
\|x_i\|_{X_i} \geq \frac{2 \|P_i\|}{1-\varepsilon} \left( \sum_{i \neq j} \|B_{ij}\| \|x_j\|_{X_j}  + \|C_i\| \|u\|_U \right).
\end{equation}
Then we obtain for all $x_i \in D(A_i)$
\[
\dot{V}_i(x_i) \leq  - \varepsilon  \|x_i\|^2_{X_i}.
\]
To verify this inequality for all $x_i \in X_i$ we use the same argument, as in the end of the proof of the Theorem \ref{LinearisationSatz} (here we use analyticity of a semigroup).

This proves, that $V_i$ is an ISS-Lyapunov function for $i$-th subsystem.
The condition \eqref{Gains_Wahl} can be easily transformed to the form \eqref{GainImplikationNSyst}, which is needed in order to apply the small-gain theorem. 

%
%
%


As a particular example consider the following system of interconnected linear reaction-diffusion equations
\begin{equation}
\label{GekoppelteLinSyst}
\left
\{
\begin {array} {l} 
{\frac{\partial s_1}{\partial t} = c_1 \frac{\partial^2 s_1}{\partial x^2} + a_{12} s_2 , \quad x \in (0,d),\ t>0,}\\
{s_1(0,t) = s_1(d,t)=0;} \\
{\frac{\partial s_2}{\partial t} = c_2 \frac{\partial^2 s_2}{\partial x^2} + a_{21} s_1 , \quad x \in (0,d),\ t>0,}\\
{s_2(0,t) = s_2(d,t)=0.} 
\end {array}
\right.	
\end{equation}
Here $c_1$ and $c_2$ are positive constants.

We choose the state space as $X_1=X_2=L_2(0,d)$.
The operators $A_i= c_i \frac{d^2}{dx^2}$ with $D(A_i)=H^1_0(0,d) \cap H^2(0,d)$, $i=1,2$ are generators of the analytic semigroups for the corresponding subsystems.

Both subsystems are ISS, moreover, $Spec(A_i) = \{ - c_i \left( \frac{\pi n}{d} \right)^2 | \ n =1,2,\ldots \}$, $i=1,2$.

Take $P_i = \frac{1}{2c_i} \left( \frac{d}{\pi} \right)^2 I$, where $I$ is the identity operator on $X_i$.
We have
\begin{eqnarray*}
\lel A_i s ,P_i s \rir + \lel P_i s,A_i s \rir &=& \frac{1}{c_i} \left(\frac{d}{\pi} \right)^2 \lel A_is, s \rir \\
   &= &  \left( \frac{d}{\pi} \right)^2    \int_0^d \frac{\partial^2 s}{\partial x^2} s dx  
		=  - \left( \frac{d}{\pi} \right)^2    \int_0^d  \left( \frac{\partial s}{\partial x} \right)^2 dx \\   &\leq & - \|s\|^2_{L_2(0,d)}. 
\end{eqnarray*}
In the last estimate we have used the Friedrichs' inequality (see p. 67 in \cite{MPF91}).
The Lyapunov functions for subsystems are defined by 
\[
V_i(s_i) = \lel P_i s_i,s_i \rir = \frac{1}{2c_i} \left( \frac{d}{\pi} \right)^2 \|s_i\|^2_{L_2(0,d)}, \mbox{ for } s_i \in X_i. 
\]
We have the following estimates for derivatives
\[
\dot{V}_1(s_1)  \leq -\|s_1\|^2_{L_2(0,d)} + \frac{1}{c_1}\left( \frac{d}{\pi} \right)^2 |a_{12}|  \|s_1\|_{L_2(0,d)}\|s_2\|_{L_2(0,d)},
\]
\[
\dot{V}_2(s_2) \leq -\|s_2\|^2_{L_2(0,d)} + \frac{1}{c_2}\left( \frac{d}{\pi} \right)^2 |a_{21}|  \|s_1\|_{L_2(0,d)}\|s_2\|_{L_2(0,d)}.
\]
We choose the gains in the following way
\[
\gamma_{12}(r)=\frac{c_2}{c_1^3}\left(\frac{d}{\pi} \right)^4 \left|\frac{a_{12}}{1-\varepsilon}\right|^2 \cdot r,
\quad
\gamma_{21}(r)=\frac{c_1}{c_2^3} \left( \frac{d}{\pi} \right)^4 \left|\frac{a_{21}}{1-\varepsilon}\right|^2 \cdot r.
\]
We have
\begin{eqnarray*}
V_1(s_1) \geq \gamma_{12} \circ V_2(s_2) &\Iff& \sqrt{ \frac{c_1}{c_2} \gamma_{12}(1) }\|s_2\|_{L_2(0,d)} \leq \|s_1\|_{L_2(0,d)}  \\
& \Iff & \frac{1}{c_1}\left( \frac{d}{\pi} \right)^2 |a_{12}| \|s_2\|_{L_2(0,d)} \leq
(1-\eps)\|s_1\|_{L_2(0,d)}.
\end{eqnarray*}
Analogously,
\begin{eqnarray*}
V_2(s_2) \geq \gamma_{21} \circ V_1(s_1) \Iff \frac{1}{c_2}\left( \frac{d}{\pi} \right)^2 |a_{21}| \|s_1\|_{L_2(0,d)} \leq (1-\eps)\|s_2\|_{L_2(0,d)}.
\end{eqnarray*}
We have the following implications:
\[
V_1(s_1) \geq \gamma_{12} \circ V_2(s_2) \quad \Rightarrow \quad \dot{V}_1(s_1) \leq - \varepsilon \|s_1\|^2_{L_2(0,d)},
\]
\[
V_2(s_2) \geq \gamma_{21} \circ V_1(s_1) \quad \Rightarrow \quad \dot{V}_2(s_2)  \leq  - \varepsilon \|s_2\|^2_{L_2(0,d)}.
\]

The small-gain condition for the case of two interconnected systems can be equivalently written as $\gamma_{12} \circ \gamma_{21} < \Id$ (see \cite{DRW07}, p. 108).
\[
\gamma_{12} \circ \gamma_{21} < \Id \quad  \Iff \quad 
\frac{1}{c_1^2 c_2^2}\left( \frac{d}{\pi} \right)^8 \frac{\left|a_{12} a_{21} \right|^2}{(1-\varepsilon)^4} <1,
\]
for arbitrary $\eps>0$. Thus, if
\begin{equation}
\label{StabBedin_LinDiff}
|a_{12} a_{21}| <c_1c_2\left( \frac{\pi}{d} \right)^4
\end{equation}
is satisfied, then the whole system \eqref{GekoppelteLinSyst} is 0-UGAS$x$.

\subsection{Nonlinear example}

Let us show the applicability of our small-gain theorem to nonlinear systems.

\begin{equation}
\label{GekoppelteNonLinSyst}
\left
\{
\begin {array} {l} 
{\frac{\partial s_1}{\partial t} = c_1 \frac{\partial^2 s_1}{\partial x^2} + s^2_2 , \quad x \in (0,d),\ t>0,}\\
{s_1(0,t) = s_1(d,t)=0;} \\
{\frac{\partial s_2}{\partial t} = c_2 \frac{\partial^2 s_2}{\partial x^2} -b s_2+ \sqrt{|s_1|} , \quad x \in (0,d),\ t>0,}\\
{s_2(0,t) = s_2(d,t)=0.} 
\end {array}
\right.	
\end{equation}
We assume, that $c_1,c_2,b$ are positive constants.

Thus, we choose the state space and space of input values for the first subsystem as $X_1=L_2(0,d)$, $U_1=L_4(0,d)$ and for the second subsystem as $X_2=L_4(0,d)$, $U_2=L_2(0,d)$. The state of the whole system \eqref{GekoppelteNonLinSyst} is denoted by $X=X_1 \times X_2$.

Define operators $B_i=c_i \frac{d^2 }{dx^2}$. These operators (together with Dirichlet boundary conditions) generate an analytic semigroup on $L_2(0,d)$ and $L_4(0,d)$ respectively (see, e.g. \cite[Chapter 7]{Paz83}).

For both subsystems take the set of input functions as $U_{c,i}:=C([0,\infty),U_i)$.
We consider the mild solutions of the subsystems, i.e. the solutions $s_i$, given by the formula \eqref{LinEq_IntegralForm}. 

Note that $s_2 \in C([0,\infty),L_4(0,d))$ $\Iff$ $s^2_2 \in C([0,\infty),L_2(0,d))$ and
$s_1 \in C([0,\infty),L_2(0,d))$ $\Iff$ $\sqrt{s_1} \in C([0,\infty),L_4(0,d))$.

Under made assumptions the solution of the first subsystem (when $s_2$ is treated as input) belongs to $C([0,\infty),H^1_0(0,d) \cap H^2(0,d)) \subset C([0,\infty),L_2(0,d))$ and the solution of the second one belongs to \\
$C([0,\infty),W^{4,1}_0(0,d) \cap W^{4,2}(0,d)) \subset C([0,\infty),L_4(0,d))$.
This implies, that the solution of the whole system is from the space $C([0,T],X)$ for all $T$ such that the solution of the whole system exists on $[0,T]$. The existence and uniqueness of the solution for all times will be proved for the values of parameters which establish ISS of the whole system, since this excludes the possibility of the blow-up phenomena.


%
%

Both subsystems of \eqref{GekoppelteNonLinSyst} are ISS. We choose $V_i$, $i=1,2$ defined by
\[
V_1(s_1)=\int_0^d{s_1^2(x)dx} = \|s_1\|^2_{L_2(0,d)}, \quad
V_2(s_2)=\int_0^d{s_2^4(x)dx} = \|s_2\|^4_{L_4(0,d)}
\]
as ISS-Lyapunov functions for $i$-th subsystem.

Consider the Lie derivative of $V_1$:
\begin{eqnarray*}
\frac{d}{dt}V_1(s_1) &=& 2 \int_0^d{s_1(x,t) \left(c_1 \frac{\partial^2 s_1}{\partial x^2}(x,t) + s^2_2(x,t) \right)dx} \\
&\leq& 
-2c_1 \left\|\frac{ds_1}{dx} \right\|^2_{L_2(0,d)} + 2\|s_1\|_{L_2(0,d)} \|s_2\|^2_{L_4(0,d)}.
\end{eqnarray*}
In the last estimation we have used the Cauchy-Schwarz inequality. By the Friedrichs' inequality, we obtain the estimation
\begin{eqnarray*}
\frac{d}{dt}V_1(s_1) & \leq & -2c_1 \left( \frac{\pi}{d} \right)^2 \|s_1\|^2_{L_2(0,d)} + 2\|s_1\|_{L_2(0,d)} \|s_2\|^2_{L_4(0,d)} \\
&= &
-2c_1 \left( \frac{\pi}{d} \right)^2 V_1(s_1) + 2 \sqrt{V_1(s_1)} \sqrt{V_2(s_2)}.
\end{eqnarray*}
Take
\[
\chi_{12}(r)= \frac{1}{c^2_1 \left( \frac{\pi}{d} \right)^4 (1-\eps_1)^2} r,\; \forall 
r>0,
\]
where $\eps_1 \in (0,1)$ is an arbitrary constant.
We obtain 
\[
V_1(s_1) \geq \chi_{12}(V_2(s_2)) \quad \Rightarrow \quad \frac{d}{dt}V_1(s_1) \leq 
-2\eps_1 c_1 \left( \frac{\pi}{d} \right)^2 V_1(s_1).
\]
Consider the Lie derivative of $V_2$:
\begin{eqnarray*}
\frac{d}{dt}V_2(s_2) &=& 4 \int_0^d{s_2^3(x,t)\left(c_2 \frac{\partial^2 s_2}{\partial x^2}(x,t) - b s_2(x,t) + \sqrt{|s_1(x,t)|}\right)dx} \\
&\leq&
- 12 c_2 \int_0^d{s_2^2 \left( \frac{\partial s_2}{\partial x} \right)^2dx} - 4 b V_2(s_2) +4 \int_0^d{ s_2^3(x,t) \sqrt{|s_1(x,t)|} dx}
\end{eqnarray*}
Applying for the last term the H\"older inequality we obtain
\[
\frac{d}{dt}V_2(s_2) \leq  
- 4b V_2(s_2) +4 (V_2(s_2))^{3/4}(V_1(s_1))^{1/4}.
\]
Let
\[
\chi_{21}(r)= \frac{1}{b^4 (1-\eps_2)^4} r,\; \forall r>0,
\]
where $\eps_2 \in (0,1)$ is an arbitrary constant. It holds the implication
\[
V_2(s_2) \geq \chi_{21}(V_1(s_1)) \quad \Rightarrow \quad 
\frac{d}{dt}V_2(s_2) \leq 
-4b\eps_2 V_2(s_2).
\]
The small-gain condition leads us to the following condition on parameters of the system
\[
\chi_{12} \circ \chi_{21} < \Id \quad  \Iff \quad 
c^2_1 \left( \frac{\pi}{d} \right)^4 (1-\eps_1)^2 b^4 (1-\eps_2)^4 > 1 
\quad  \Iff \quad 
c_1 \left( \frac{\pi}{d} \right)^2 b^2 >1.
\]
This condition guarantees that the system \eqref{GekoppelteNonLinSyst} is 0-UGAS$x$.

Note that the above stability condition doesn't involve the parameter $c_2$, and it provides good estimate for the stability region of the system if $c_2$ is small. Otherwise more precise analysis must be made.

\section{Conclusion}
\label{Schluss}

In this paper we have performed several steps towards generalization of the ISS theory to infinite-dimensional systems. The developed
framework which encompasses the ODE systems, systems with time-delays as well as many classes of evolution PDEs and is consistent with the original definitions of ISS for ODEs and time-delay systems.

In Section~\ref{LyapFunk} we have proved, that existence of an ISS-Lyapunov function implies the ISS property of a general control system and we have shown, how our definition of the ISS-Lyapunov function reduces
to the standard one in the case of finite-dimensional systems. For the systems, governed by differential equations in Banach spaces we established in Section~\ref{GekoppelteISS_Systeme} a small-gain theorem,
which provides us with a design of an ISS-Lyapunov function for an interconnection of ISS subsystems, provided the ISS-Lyapunov functions for the subsystems are known and a small-gain condition holds.

For constructions of local ISS-Lyapunov functions the linearization method has been proposed in Section~\ref{Linearisierung}, which is a good alternative to Lyapunov methods provided the system is linearizable.

Many interesting problems remain open.
For example, the most part of results in this work as well as in several other papers on ISS theory of infinite-dimensional systems \cite{MaP11}, \cite{PrM12}, \cite{KaJ11},
have been proved for either piecewise-continuous or continuous inputs. This can be quite restrictive for many applications, in particular, for PDEs and requires further research.

Another important problem is to prove (or disprove) the characterizations of ISS for infinite-dimensional systems analogous to the ones developed by Sontag and Wang in \cite{SoW95} and \cite{SoW96}
for finite-dimensional systems (with $X=\R^n$ and $U_c=L_{\infty}(\R_+,\R^m)$). The converse Lyapunov theorem is another desired fundamental theoretical result, which is beyond the scope of this paper.

\begin{acknowledgements}
This research is funded by the German Research Foundation (DFG) as a part of Collaborative Research Centre 637 "Autonomous Cooperating Logistic Processes - A Paradigm Shift and its Limitations". 

We thank anonymous reviewers for their comments and suggestions, which led to improvements in the presentation of results.
\end{acknowledgements}

%


%

\end{document}